\RequirePackage{fix-cm}
\documentclass[12pt]{article}
\usepackage{authblk}
\usepackage{hyperref}
\usepackage[cmintegrals,cmbraces]{newtxmath}
\usepackage[left=2.4cm,right=1cm]{geometry}
\usepackage{setspace}

\usepackage{amsthm}
\usepackage{lmodern}
\usepackage{libertinus}
\usepackage{latexsym}
\usepackage{tikz-cd}
\usepackage{graphicx} 
\usetikzlibrary{cd}
\usepackage{amscd}
\usepackage[all]{xy}
\usepackage{breqn}
\usepackage{tcolorbox}
\usepackage{biblatex}
\usepackage{footmisc}
\usepackage{comment}
\usepackage{thmtools}

\declaretheorem[name=Theorem,numberwithin=section]{theorem}
\declaretheorem[name=Proposition,sibling=theorem]{proposition}
\declaretheorem[name=Lemma,sibling=theorem]{lemma}
\declaretheorem[name=Corollary,sibling=theorem]{corollary}
\declaretheorem[name=Example,numberwithin=section]{example}
\declaretheorem[name=Remark,sibling=example]{remark}

\newcommand{\lra}{\leftrightarrows}
\newcommand{\ra}{\rightarrow}

\newcommand{\da}{{\downarrow}}
\newcommand{\up}{{\uparrow}}

\newcommand{\mb}[1]{\mbox{#1}}
\newcommand{\ca}[1]{\mathcal{#1}}
\newcommand{\bd}[1]{\mathbf{#1}}
\newcommand{\mf}{\mathsf}

\newcommand{\mi}{\mathit}

\newcommand{\se}{\subseteq}
\newcommand{\sm}{\setminus}

\newcommand{\we}{\wedge}
\newcommand{\ve}{\vee}
\newcommand{\bwe}{\bigwedge}
\newcommand{\bve}{\bigvee}
\newcommand{\bca}{\bigcap}
\newcommand{\bcu}{\bigcup}
\newcommand{\opp}[1]{#1^{op}}

\newcommand{\Sc}{\mf{S}_{\cl}}
\newcommand{\So}{\mf{S}_{\op}}
\newcommand{\Ss}{\mf{S}}
\newcommand{\Se}{\mf{S}_{\ca{E}}}

\newcommand{\op}{\mathfrak{o}}
\newcommand{\cl}{\mathfrak{c}}
\newcommand{\bl}{\mathfrak{b}}

\newcommand{\Fi}{\mathsf{Filt}}
\newcommand{\fso}{\Fi_{\mathcal{SO}}}
\newcommand{\fcp}{\Fi_{\mathcal{CP}}}
\newcommand{\fr}{\Fi_{\mathcal{R}}}
\newcommand{\fe}{\Fi_{\mathcal{E}}}
\newcommand{\fse}{\Fi_{\mathcal{SE}}}

\newcommand{\Om}{\Omega}
\newcommand{\pt}{\mathsf{pt}}
\newcommand{\rpt}{\mathsf{pt}_R}
\newcommand{\maxpt}{\mathsf{maxpt}}

\newcommand{\sll}{\mf{S}(L)}

\newcommand{\scl}{\mf{S}_{\cl}(L)}

\title{Raney extensions: a pointfree theory of $T_0$ spaces based on canonical extension}
\author{Anna Laura Suarez \thanks{Department of Mathematics Tullio Levi-Civita, University of Padova, 35121, Italy. email: \href{mailto:annalaurasuarez993@gmail.com}{annalaurasuarez993@gmail.com}}}
\date{}

\begin{document}

\maketitle

\vspace{-18pt}

\begin{abstract}
\begin{spacing}{0.95}
We introduce a pointfree version of Raney duality. Our objects are \emph{Raney extensions} of frames, pairs $(L,C)$ where $C$ is a coframe and $L\se C$ is a subframe that meet-generates it and whose embedding preserves strongly exact meets. We show that there is a dual adjunction between $\bd{Raney}$ and $\bd{Top}$, with all $T_0$ spaces as fixpoints, assigning to a space $X$ the pair $(\Om(X),\ca{U}(X))$, with $\ca{U}(X)$ are the intersections of open sets. We show that for every Raney extension $(L,C)$ there are subcolocale inclusions $\opp{\Sc(L)}\se C\se \So(L)$ where $\mf{S}_{\op}(L)$ is the coframe of fitted sublocales and $\Sc(L)$ is the frame of joins of closed sublocales. We thus exhibit a symmetry between these two well-studied structures in pointfree topology. The spectra of these are, respectively, the classical spectrum $\pt(L)$ of the underlying frame and its $T_D$ spectrum $\pt_D(L)$. This confirms the view advanced in \cite{banaschewskitd} that sobriety and the $T_D$ property are mirror images of each other, and suggests that the symmetry above is a pointfree view of it. All Raney extensions satisfy some variation of the properties \emph{density} and \emph{compactness} from the theory of canonical extensions. We characterize sobriety, the $T_1$, and the $T_D$ axioms in terms of density and compactness of $(\Om(X),\ca{U}(X))$. We characterize frame morphisms $f:L\to M$ that extend to Raney morphisms $\overline{f}:(L,C)\to (M,D)$. We thus obtain a characterization of morphisms of frames $f:L\to M$ which extend to frame morphisms $\overline{f}:\Sc(L)\to \Sc(M)$, answering a question posed in \cite{ball14b}. We show the existence of the free Raney extension over a frame. We show that all Raney extensions admit a sober coreflection. Restricting morphisms of $\bd{Raney}$ to \emph{exact} morphisms gives both cofree objects and $T_D$ reflections. Finally, we show that the canonical extension of a locally compact frame (introduced in \cite{jakl20}) is the free algebraic Raney extension. We also give a new view of $T_D$ duality: in contrast with the frame case, $T_D$ spaces are are full subcategory of $\bd{Raney}$, with no need to restrict morphisms.
\end{spacing}
\end{abstract}

\tableofcontents

\section{Introduction}

In this work, we study an algebraic, pointfree version of the embedding of a frame of open sets into a lattice of saturated sets. Our constructions are inspired by \emph{Raney duality}, as illustrated in \cite{bezhanishvili20}. In Raney duality, we map a space $X$ to the pair $(\Om(X), \ca{U}(X))$ where $\ca{U}(X)$ is the lattice of saturated sets\footnote{\emph{Saturated sets} are intersections of open sets.}. In this duality, on the algebraic side all objects of the category have, so to speak, enough points. Every Raney algebra is of the form $(\Om(X),\ca{U}(X))$ for some space $X$. In this work, we extend the category of Raney algebras to include pointfree objects. We do so by taking as objects generalizations of the \emph{canonical extension} construction for distributive lattices. Canonical extensions for Boolean algebras were introduced by Jónsson and Tarski (see \cite{jonsson51i} and \cite{jonsson51ii}) in dealing with Boolean algebras with operators. They have later proven to be useful in theoretical Computer Science and Logic, and have been generalized to various settings. Canonical extension may be seen as an algebraic and pointfree version of the lattice of compact opens into the lattice of saturated sets of a Stone space. Canonical extensions have also been introduced for distributive lattices. On this topic, we refer the reader to \cite{gehrke94}, \cite{gehrke01}, and \cite{gehrke08}. For distributive lattices this represents the embedding of the lattice of compact open sets of a coherent space into the lattice of saturated sets. For a distributive lattice, its canonical extension is unique. In \cite{jakl20}, canonical extension is introduced for locally compact frames. For general frames, we claim, there is no unique way of extending a frame $L$ to a pointfree lattice of saturated sets. For a concrete space $X$, the structure of $\ca{U}(X)$ will depend, for example, on how many of the points of $\Om(X)$ are realized as concrete points of $X$.

\medskip

In Section \ref{S1} we introduce the main structures. We will consider as pointfree spaces \emph{Raney extensions}, pairs $(L,C)$ where $C$ is a coframe and $L\se C$ is a frame which meet-generates $C$ and such that the embedding preserves the frame operations together with strongly exact meets\footnote{\emph{Strongly exact meets} may be seen as the pointfree version of those intersections of open sets which are open, as they are characterized as those that give collections of open sublocales with open intersections. Because a meet of a collection $\{U_i:i\in I\}$ of opens in general is calculated as the interior of $\bca_i U_i$, these are exactly the meets that are preserved by the embedding $\Om(X)\se \ca{U}(X)$. }. Every Raney extension will satisfy a generalization of the property \emph{density} and \emph{compactness} from the theory of canonical extensions. We build on the work in \cite{jakl24} connecting canonical extensions and pointfree topology. In pointfree topology, a structure of primary importance is the structure $\sll$ of sublocales (pointfree subspaces) of a frame $L$. The structures $\mf{S}_{\cl}(L)$ of joins of closed sublocales and $\mf{S}_{\op}(L)$ of fitted sublocales have been widely studied, and these are compared in \cite{moshier20} and in \cite{moshier22}. We will show that these are the largest and smallest Raney extensions on a frame $L$, respectively.

\medskip

In Section \ref{S2} we will show that there is an adjunction $\Om_R:\bd{Top}\lra \bd{Raney}^{op}:\pt_R$ extending Raney duality. As for canonical extensions, the points of a Raney extension $(L,C)$ are completely join-prime elements of $C$. We compute the spectrum of the largest and smallest Raney extensions $(L,\So(L))$ and $(L,\opp{\Sc(L)})$, and discover that these are, respectively, the spectrum $\pt(L)$ of the underlying frame and its $T_D$ spectrum $\pt_D(L)$. The $T_D$ spectrum was introduced in \cite{banaschewskitd}, where a duality is shown between $T_D$ spaces and the category $\bd{Frm}_D$, obtained by restricting morphisms in $\bd{Frm}$. In our case, $T_D$ spaces are faithfully represented in $\bd{Raney}$ as the subcategory of the (spatial) Raney extensions of the form $(L,\opp{\Sc(L)})$, with no need to restrict morphisms.

\medskip 

In Section \ref{S3} we exploit the duality to characterize topological properties such as sobriety, the $T_D$ and the $T_1$ axioms in terms of density and compactness of their Raney extensions. In pointfree topology, the $T_1$ axiom has several different translations; weakest is \emph{subfitness}, see for example Chapter V of \cite{picadopultr2012frames} or Chapter II of \cite{picado21}. A Raney extension $(L,C)$ is defined to be $T_1$ if $C$ is Boolean, just like a space $X$ is $T_1$ if and only if $\ca{U}(X)$ is the powerset. We connect the two views by showing that a frame is subfit if and only if it admits a $T_1$ Raney extension. We thus exhibit subfitness as the weakest possible frame version of the $T_1$ axiom. The notions of sober and $T_D$ objects have no counterpart in the frame setting. Another axiom which has been studied quite extensively is \emph{scatteredness}. Scatteredness for a frame $L$ is defined in \cite{plewe00} and \cite{plewe02} as the property that $\mf{S}(L)$ is Boolean. In \cite{ball16}, the authors characterize the frames for which $\mf{S}_{\cl}(L)=\mf{S}(L)$ as those subfit frames such that they are scattered. Here, we show that subfit frames which are scattered coincide with those subfit frames with unique Raney extensions.

\medskip

In Section \ref{S4}, we will characterize frame morphisms $f:L\to M$ that extend to Raney morphisms $\overline{f}:(L,C)\to (M,D)$. In this case, the requirement is weaker than in the theory of canonical extension, where extensions of morphisms are required to preserves all meets and all joins.  We uswe this to answer the question posed in \cite{ball19} of what morphisms of frames $f:L\to M$ lift to frame morphisms $\Sc(f):\Sc(L)\to \Sc(M)$. We call them \emph{exact} frame morphisms. We also use the result to show that every frame admits a free Raney extensions on it. Restricting morphisms on $\bd{Raney}$ suitably also gives $T_D$ reflections. We also show that every Raney extension has a sober coreflection. Both these results have no counterpart in $\bd{Frm}$. 

\medskip

In Section \ref{S5}, finally, we look at two special topics. We will look at canonical extensions of frames, as defined in \cite{jakl20}, and show that for a pre-spatial frame its canonical extension is the free Raney extension over it which is \emph{algebraic} (join-generated by its compact elements). We will also show that a $T_0$ space is sober, resp. strictly sober, if and only if the Raney extension $(\Om(X),\ca{U}(X))$ is compact with respect to the collection of completely prime, resp. Scott-open, filters. We also look at the category of frames with exact morphisms $\bd{Frm}_{\ca{E}}$. We show that this is isomorphic to the subcategory $\bd{Raney_D}$ of $\bd{Raney}$, defined as $\{(L,\opp{\Sc(L)}):L\in \bd{Frm}\}$. This gives a new view of $T_D$ duality, pictured below. Note that $\Sc:\bd{Frm}_{\ca{E}}\to \bd{Raney_D}$ is an equivalence.

\begin{center}
    \begin{tikzcd}[row sep=large,column sep=large]
        \bd{Frm}_{\ca{E}}
        \ar[r,"\pt_D"]
        \ar[d,"\Sc"]
        & \bd{Top}_D
        \\
       \bd{Raney_D},
    \ar[ur,"\rpt"]
    \end{tikzcd}
\end{center}

\section{Background}

\subsection{Sublocales}

We will work in the category $\bd{Frm}$. Sometimes in pointfree topology one works in the category $\bd{Loc}$ of locales. The category $\bd{Loc}$ is defined as the category whose objects are frames, which are referred to as \emph{locales} when adopting this approach. The morphisms of $\bd{Loc}$ are the right adjoints to frame maps. Hence, a frame map $f:L\ra M$ will correspond to the morphism $f_*:M\ra L$ in $\bd{Loc}$. The category $\bd{Loc}$ is dually isomorphic to $\bd{Frm}$. In the category of topological spaces, subspace inclusions are, up to isomorphism, the regular monomorphisms. A \emph{sublocale} is a regular monomorphism in $\bd{Loc}$. Even when working with frames, the term \emph{sublocale} is still used. We follow Picado and Pultr in \cite{picadopultr2012frames} in defining a \emph{sublocale} of a frame $L$ to be a subset $S\se L$ such that: 
\begin{enumerate}
    \item It is closed under all meets;
    \item Whenever $s\in S$ and $x\in L$, $x\ra s\in S$.
\end{enumerate}

These requirements are equivalent to stating that $S\se L$ is a regular monomorphism in $\bd{Loc}$. Observe that the collection of sublocales of a frame is closed under all intersections. The following is a useful fact.
\begin{lemma}\label{slofsl}
If $S$ and $T$ are sublocales of $L$ such that $S\se T$, then $S$ is a sublocale of $T$.
\end{lemma}

The family $\mathsf{S}(L)$ of all sublocales of $L$ ordered by inclusion is a coframe. Meets in $\mf{S}(L)$ are set-theoretical intersections.  For a subset $X\se L$, we denote as $\ca{S}(X)$ the smallest sublocale containing $X$. In the following, $\ca{M}(-)$ denotes closure under meets.
\begin{lemma}\label{l:small-sl}
    For a frame $L$ and for $X\se L$, $\ca{S}(X)=\ca{M}(\{a\ra x:a\in L,x\in X\})$.
\end{lemma}

The top element is $L$ and the bottom element is $\{1\}$. Because $\sll$ is a coframe, there is a \emph{difference} operator on it, dual to Heyting implication, defined for sublocales $S$ and $T$ as $S{\sm}T=\bca \{U\in \sll:S\se T\cup U\}$. For a sublocale $S$, we denote the element $L{\sm}S$ as $S^*$, and we call it the \emph{supplement} of $S$. For each $a\in L$, there are an \emph{open sublocale} and a \emph{closed sublocale} associated with it. These are, respectively, $\mathfrak{o}(a)=\{a\ra b: b\in L\}$ and $\mathfrak{c}(a)=\uparrow a$. We will need a few facts about open and closed sublocales, which we gather here.
\begin{proposition}\label{manyfacts}
For every frame $L$ and $a,b,a_i\in L$:
\begin{enumerate}
    \item $\op(1)=L$ and $\op(0)=\{1\}$;
    \item $\cl(1)=\{1\}$ and $\cl(0)=L$;
    \item $\bve_i \op(a_i)=\op(\bve_i a_i)$ and $\op(a)\cap \op(b)=\op(a\we b)$;
    \item $\bca_i \cl(a_i)=\cl(\bwe_i a_i)$ and $\cl(a)\ve \cl(b)=\cl(a\we b)$;
    \item The elements $\op(a)$ and $\cl(a)$ are complements of each other in $\mf{S}(L)$;
    \item $\cl(a)\se \op(b)$ if and only if $a\ve b=1$, and $\op(a)\se \cl(b)$ if and only if $a\we b=0$.
\end{enumerate}
\end{proposition}
Every sublocale can be written as an intersection of sublocales of the form $\op(x)\ve \cl(y)$. A sublocale is \emph{fitted} if it is an intersection of open sublocales. We call $\So(L)$ the ordered collection of all fitted sublocales. We call $\Sc(L)$ the ordered collection of joins of closed sublocales. This collection is studied in \cite{joinsofclosed19}. For a coframe $C$, we say that an element $c\in C$ is \emph{linear} if $\bve_i (x_i \we c)=\bve_i x_i \we c$ for any collection $x_i\in C$.

\begin{lemma}\label{l:linear}
Complemented elements of a coframe are linear. In particular, in $\sll$ open and closed sublocales are linear.
\end{lemma}

 Also particularly important are \emph{Boolean sublocales}. For an element $a\in L$ the sublocale $\{x\to a:x\in L\}$, denoted as $\bl(a)$, is the smallest sublocale containing $a$. A sublocale is a Boolean algebra if and only if it is of this form for some $a\in L$. An element $p\in L$ is \emph{prime} when $x\we y\leq p$ implies either $x\leq p$ or $y\leq p$, for all $x,y\in L$. Elements of the form $\bl(p)$ are also called \emph{two-element sublocales}, as for $p$ prime $\bl(p)=\{1,p\}$.

\begin{lemma}\label{l:prime}
For a frame $L$, the following hold for each prime $p\in L$ and all elements $x,y\in L$.
\begin{itemize}
    \item $x\to p=1$ if $x\leq p$, and $x\to p=p$ if $x\nleq p$.
    \item $\bl(p)\se \op(x)$ if and only if $x\nleq p$.
    \item The element $\bl(p)$ is completely join-prime in $\mf{S}(L)$.
    \item The prime elements of $\mf{S}(L)$ are the sublocales of the form $\bl(p)$ for some prime $p\in L$.
\end{itemize}
    
\end{lemma}

\subsection{Saturated sets and fitted sublocales}
In a topological space $X$, we can define the \emph{specialization preorder} on its points, defined as $x\leq y$ whenever $x\in U$ implies $y\in U$ for all open sets $U\se X$. In this paper, for a space $X$, we will denote as $\ca{U}(X)$ the ordered collection of all upper sets in the specialization preorder, and for a point $x\in X$ we denote as $\up x$ the upper set of $x$ with respect to this preorder. A space $X$ is $T_0$ if and only if the specialization preorder is an order. A space is $T_1$ if and only if the specialization order on it is discrete. For a space $X$, we denote as $\ca{U}(X)$ the lattice of its upsets (upper-closed sets) under the specialization preorder. The following is a standard fact of topology, which can be easily checked.
\begin{proposition}\label{upset}
For a topological space $X$, a subset is an upset in the specialization preorder if and only if it is saturated.
\end{proposition}
 The following is an important theorem by Hofmann and Mislove.  A filter of a frame $L$ is \emph{Scott-open} if it is not accessible by directed joins. We call $\fso(L)$ the ordered collection of Scott-open filters of a frame $L$.

\begin{theorem}(\cite{hofmann81}, Theorem 2.16)\label{hofmis}
    If the Prime Ideal Theorem holds, then for each sober space $X$ there is an anti-isomorphism between $\fso(\Om(X))$ and the ordered collection of compact saturated sets of $X$, assigning to each filter $F$ the set $\bca F$.
\end{theorem}

\subsubsection*{Exact and strongly exact filters}
Recall that a meet $\bwe_i x_i$ is \emph{strongly exact} if, for all $y\in L$, $x_i\ra y=y$ implies $(\bwe_i x_i)\ra y=y$, and that a filter is \emph{strongly exact} if it is closed under strongly exact meets. We call $\fse(L)$ the ordered collection of strongly exact filters. This is a frame where meets are computed as intersections, and additionally it is a sublocale of $\mf{Filt}(L)$.  A meet $\bwe_i x_i$ of a frame $L$ is \emph{exact} if, for every $a\in L$, $(\bwe_i x_i)\ve a=\bwe_i (x_i\ve a)$. \emph{Exact} filters are those closed under exact meets. They form a frame, and in particular the frame $\fe(L)$ of exact filters is a sublocale of $\fse(L)$. The main theorem that we will need is the following. In the following, $\mi{ker}$ stands for \emph{kernel}, and $\mi{coker}$ for \emph{cokernel}. The following two results, which we state as one theorem, are shown in \cite{moshier20} and \cite{ball20}, respectively.

\begin{theorem}\label{eandse}
    There is an isomorphism of coframes
    \begin{gather*}
        \mi{ker}:\mf{S}_{\op}(L)\cong \opp{\fse(L)},\\
        S\mapsto \{a\in A:S\se \op(a)\}.
    \end{gather*}
    We also have an isomorphism of frames
    \begin{gather*}
         \mb{coker}:\mf{S}_{\cl}(L)\cong\fe(L),\\
         S\mapsto \{a\in L:\cl(a)\se S\}.
    \end{gather*}
\end{theorem}

The following results give topological intuition on exactness and strong exactness.

\begin{proposition}\label{p:pres-se}
    For every space $X$, strongly exact meets in $\Om(X)$ are open sets. This means that the embedding $\Om(X)\se \ca{U}(X)$ preserves strongly exact meets.
\end{proposition}

\begin{theorem}[see \cite{ball14}, Theorem 5.2.3]\label{t:ball}
    A $T_0$ space is $T_D$ if and only if for every exact meet $\bwe_i U_i$ in $\Om(X)$ this equals $\bca_i U_i$. This is equivalent to the embedding $\Om(X)\se \ca{U}(X)$ preserving exact meets.
\end{theorem}

\subsubsection*{Notable collections of filters}
In this paragraph, we refer to \cite{jakl24}, and mention the main results from that we are going to use. We will refer to several important concrete collections of filters. Since the collection $\mathsf{Filt}(L)$ is a frame, there is a Heyting operation $\ra$ on it. Notice that for a frame $L$ and for $a,b\in L$
\[
\up a\ra \up b=\{x\in L:b\leq x\ve a\}.
\]

This gives a useful characterization of exact filters.

\begin{lemma}[\cite{jakl24}, Proposition 5.5]\label{charexact}
A filter is exact if and only if it is the intersection of filters of the form $\up a\ra \up b$ for some $a,b\in L$. In particular, if $F$ is an exact filter,
\[
F=\bca \{\up a\ra \up b:b\leq a\ve f\mbox{ for all }f\in F\}.
\]
\end{lemma}
In particular, note that this means that, for any $a\in L$, $\neg \up a=\up a\to \{1\}=\{x\in L:x\ve a=1\}$. We say that a filter is \emph{regular} if it is a regular element in the frame of filters (that is, if it is of the form $\neg F$ for some filter $F$). We call $\mathsf{Filt}_{\ca{R}}(L)$ the ordered collection of regular filters. Note that $\fr(L)\se \Fi(L)$ is the Booleanization of the frame of $\Fi(L)$. Regular filters, too, have a useful concrete characterization.

\begin{proposition}[\cite{jakl24}, Lemma 5.6]\label{reg}
    The regular filters coincide with the intersections of filters of the form $\{x\in L:x\ve a=1\}$ for some $a\in L$. In particular, if $F$ is a regular filter,
    \[
    F=\{x\in L:x\ve f=1\mb{ for all }f\in F\}.
    \]
\end{proposition}

 In the following, $\fcp(L)$ is the collection of completely prime filters and $\fso(L)$ that of Scott-open filters, and $\ca{I}(-)$ denotes closure under set-theoretical intersections. Note that this includes the empty intersection, namely the whole frame $L$.
 
\begin{theorem}[\cite{jakl24}, Corollary 5.11]\label{posetofsubloc}
For any frame $L$, there is the following poset of sublocale inclusions:
\[
\begin{tikzcd}
\fr(L)
\ar[r,"\se"]
&\mathsf{Filt}_{\mathcal{E}}(L)
\ar[dr,"\se"]\\
&
&
\mathsf{Filt}_{\mathcal{SE}}(L).
\\
\ca{I}(\mathsf{Filt}_{\mathcal{CP}}(L))
\ar[r,"\se"]
& \ca{I}(\mathsf{Filt}_{\mathcal{SO}}(L))
\ar[ur,"\se"]
\end{tikzcd}
\]
\end{theorem}

Finally, the following characterizations of frame properties in terms of collections of filters.

\begin{proposition}[\cite{jakl24}, Proposition 5.12]\label{famouschar}
For a frame $L$:
\begin{itemize}
    \item $L$ is pre-spatial if and only if $\ca{I}(\fso(L))$ contains all principal filters;
    \item $L$ is spatial if and only if $\ca{I}(\fcp(L))$ contains all principal filters;
    \item $L$ is subfit if and only if $\fr(L)$ contains all principal filters.
\end{itemize}
\end{proposition}

\subsection{Canonical extensions and the Prime Ideal Theorem}
\subsubsection*{Canonical extensions}

In \cite{jakl20}, the question of what is the canonical extension of a frame is tackled for locally compact frames; there, the canonical extension of a general frame $L$ is defined as a monotone map $f^{\delta}:L\ra L^{\delta}$ to a complete lattice $L^{\delta}$ such that the following two properties hold:
\begin{enumerate}
    \item Density: every element of $L^{\delta}$ is a join of elements in $\{\bwe f[F]:F\in \fso(L)\}$;
    \item Compactness: for every Scott-open filter $F$, $\bwe f[F]\leq f(a)$ implies $a\in F$, for each $a\in L$.
\end{enumerate}

\begin{theorem}(\cite{jakl20}, Theorem 4.2)
For a frame $L$, its canonical extension is unique, up to isomorphism. This is the map
\begin{gather*}
    L\ra \opp{\ca{I}(\fso(L))},\\
    a\mapsto \bca \{F\in \fso(L):a\in F\}.
\end{gather*}
\end{theorem}

A frame $L$ is \emph{pre-spatial} if whenever $a\nleq b$ there is a Scott-open filter containing $a$ and omitting $b$, for all $a,b\in L$.
\begin{proposition}(\cite{jakl20}, Proposition 5.1)
The map $f^{\delta}:L\ra L^{\delta}$ is an injection if and only if $L$ is pre-spatial. 
\end{proposition}

\subsubsection*{The Prime Ideal Theorem, pre-spatiality, and strict sobriety}
In \cite{erne07} the Prime Ideal Theorem -- \textbf{PIT} hereon -- is shown to be equivalent to the statement that every pre-spatial frame is also spatial. The so-called \emph{Strong Prime Element Theorem} -- which we will abbreviate as $\textbf{SPET}$ -- states that for every complete distributive lattice $D$, and any Scott-open filter $F\se D$, for every element $a\in D$ not in $F$ there is a prime element $p\in D$ above $a$ with $p\notin F$. In \cite{Banaschewski93} (Proposition 1) it is shown that $\textbf{PIT}$  implies $\textbf{SPET}$. It is also known that $\textbf{SPET}$ implies $\textbf{PIT}$. 
In \cite{erne07} the notion of \emph{strict sobriety} is introduced, and it is shown that sobriety implying strict sobriety is equivalent to the Ultrafilter Principle and several others choice principles. The concept was later developed in \cite{erne18}. A space $X$ is \emph{strictly sober} if it is $T_0$ and every Scott-open filter of its frame of opens is $\{U\in \Om(X):F\se U\}$ for some saturated set $F$, which is then necessarily compact. Strict sobriety is stronger than sobriety, even without assuming any choice principles.

\subsection{\texorpdfstring{T\textsubscript{D}}{TD} duality}

 A topological space $X$ is said to be $T_D$ if for every point $x\in X$ there are opens $U$ and $V$ such that $U{\sm}V=\{x\}$. For a frame $L$ we say that a prime $p\in L$ is \emph{covered} if whenever $\bwe_i x_i=p$ for some family $x_i\in L$ then $x_i=p$ for some $i\in I$. In \cite{banaschewskitd} the \emph{$T_D$ spectrum} of a frame $L$ is defined as the collection of covered primes of a frame, with the subspace topology inherited from the prime spectrum of $L$. This space is denoted as $\pt_D(L)$. This turns out to always be a $T_D$ space. A frame morphism $f:L\to M$ is a \emph{D-morphism} if for every covered prime $p\in L$ the prime $f_*(p)$ is covered. We call $\bd{Frm}_D$ the category of frames and D-morphisms. There is a dual adjunction $\Om:\bd{Top}\lra \bd{Frm}_D:\pt_D$, where the fixpoints on the space side are the $T_D$ spaces, and on the frame side these are the \emph{D-spatial} frames, which can be characterized as those frame such that all their elements are the meet of the covered primes above them. We will use the following two results.

 \begin{proposition}\label{allpxarecov}(\cite{banaschewskitd}, Proposition 2.3.2)
    A space $X$ is $T_D$ if and only if all elements of the form $X{\sm}\overline{\{x\}}$ are covered primes in $\Om(X)$.
\end{proposition}

Furthermore, in \cite{arrieta21} the notion of \emph{D-sublocale} is introduced. This is a sublocale $S\se L$ such that the corresponding surjection is in $\bd{Frm}_D$. 
\begin{theorem}\label{t:d-subloc}
    For a frame $L$, the D-sublocales form a subcolocale $\mf{S}_D(L)\se \sll$. We also have a subcolocale inclusion $\Sc(L)\se \mf{S}_D(L)$.
\end{theorem}

\section{Raney extensions}\label{S1}

For a complete lattice $C$, we say that $L\se C$ is a \emph{subframe} of $C$ if $L$ equipped with the inherited order is a frame, and if the embedding $L\se C$ preserves all joins and finite meets. A \emph{Raney extension} is a pair $(L,C)$ such that $C$ is a coframe and $L$ is a subframe of $C$ such that:
\begin{itemize}
    \item The frame $L$ meet-generates $C$;
    \item The embedding $L\se C$ preserves strongly exact meets.
\end{itemize}. 
We will sometimes use the expression \emph{Raney extension} to refer to the coframe component of the pair, and for a pair $(L,C)$ we will say that this is a Raney extension of $L$, or that it is a Raney extension \emph{over} $L$. A morphism of Raney extensions $f:(L,C)\ra (M,D)$ is a coframe map $f:C\ra D$ such that, whenever $a\in L$, $f(a)\in M$ and such that the restriction $f|_L:L\ra M$ is a frame map. We call $\bd{Raney}$ the category of Raney extensions with Raney maps.

\begin{example}\label{example1}
For a topological space $X$, the pair $(\Om(X),\mathcal{U}(X))$ is a Raney extension. That strongly exact meets are preserved by the embedding $\Om(X)\se \ca{U}(X)$ is the content of Proposition 5.3 of \cite{ball14}.
\end{example}

Observe that for any Raney extension $(L,C)$, by the universal property of the ideal completion of a distributive lattice, there is a coframe surjection $\bwe:\opp{\Fi(L)}\to C$ extending $L\se C$. For $c\in C$, define $\up^L c$ as $\up c\cap L$. Notice that for each filter $F\in \Fi(L)$ and each $c\in C$:
\[
c\leq \bwe F\mb{ if and only if }c\leq f\mb{ for all $f\in F$, if and only if $F\se \up^L c$}.
\]
This means that $\up^L:C\to \opp{\Fi(L)}$ is left adjoint to $\bwe:\Fi(L)^{op}\to C$. As the starting map $\bwe$ is a coframe map, the inclusion of the fixpoints $C^*:=\{\up^L c\mid c\in C\}\se \Fi(L)$ is a sublocale.

\begin{theorem}\label{charC*}
For a Raney extension $(L,C)$, there is an adjunction 
\[
\bwe:\opp{\Fi(L)}\lra C:\up^L,
\]
which maximally restricts to a pair of mutually inverse isomorphisms
\[
\bwe:C^*\lra C:\up^L.
\]
These are also isomorphisms of Raney extensions $\bwe:(L,C^*)\lra (L,C):\up^L$.
\end{theorem}
\begin{proof}
 It only remains to show that the isomorphisms $\up^L:C\lra C^*:\bwe$ restrict correctly to the frame components. We notice that for $a\in L$, indeed, the filter $\up^L a$ is the principal filter $\up a\se L$, an element of the generating frame of $C^*$. Conversely, any principal filter of $L$ is of this form.
\end{proof}

Let us now tie the notion of Raney extension with that of canonical extension. For a monotone map $f:L\to C$ of a lattice $L$ into a complete lattice $C$, we introduce the following two properties:

\begin{enumerate}
    \item $\ca{F}$-density: the collection $\{\bwe f[F]:F\in \ca{F}\}$ join-generates $C$;
    \item $\ca{F}$-compactness: $\bwe f[F]\leq f(a)$ implies $a\in F$ for every $F\in \mathcal{F}$ and every $a\in L$.
\end{enumerate}

We say that the map is $\ca{F}$-canonical if and only if it is both $\ca{F}$-dense and $\ca{F}$-compact. For brevity, in the following we will refer to $\fso(L)$-canonicity simply as $\ca{SO}$-canonicity, and analogously for all other similarly denoted collections of filters, and for density and compactness.

\begin{example}
For a sober space $X$, the pair $(\Om(X),\mathcal{U}(X))$ is a $\mathcal{SO}$-canonical Raney extension, provided that the Prime Ideal Theorem holds. This is observed in Example 3.5 of \cite{jakl20}. On the one hand, the coframe $\ca{U}(X)$ is join-generated by elements of the form $\up x$ for $x\in X$, and these are intersections of neighborhood filters, which are completely prime, hence Scott-open. Then, the extension is $\ca{SO}$-dense. For $\ca{SO}$-compactness, we rely on the Hofmann-Mislove Theorem. If $F\se \Om(X)$ is a Scott-open filter, then by the Theorem it must be $\{U\in \Om(X):\bca F\se U\}$, and so, indeed, for every open $U$, $\bca F\se U$ implies $U\in F$. Recall that the Hofmann-Mislove Theorem is dependent on the Prime Ideal Theorem, see for instance \cite{erne18}, Theorem 3. With Proposition \ref{p:charsober} we will see that if we replace \emph{Scott-open} by \emph{completely prime}, there is an analogous result which does not rely on the Prime Ideal Theorem. We will explore the relation between the Prime Ideal Theorem and $\ca{SO}$-canonicity in Subsection \ref{S51}.
\end{example}

Theorem \ref{charC*} tells us that for a Raney extension $(L,C)$ we may identify elements of $C$ with filters of $L$. In the following, for a collection of filters $\ca{F}$, we denote as $\ca{F}^*$ the collection $\{\up^L \bwe F:F\in \ca{F}\}$.
\begin{proposition}\label{C*canonical}
    For any Raney extension $(L,C)$ and any collection $\ca{F}\se \Fi(L)$,
    \begin{enumerate}
        \item $(L,C)$ is $\ca{F}$-dense if and only if $C^*\se \ca{I}(\ca{F}^*)$;
        \item $(L,C)$ is $\ca{F}$-compact if and only if $\ca{F}\se C^*$.
    \end{enumerate}
    In particular, $(L,C)$ is $\ca{F}$-canonical if and only if $\ca{I}(\ca{F})^{op}=C^*$.
\end{proposition}
\begin{proof}
Let us prove the first claim. If $(L,C)$ is $\ca{F}$-dense, then, for all $c\in C$, $c=\bve_i \bwe F_i$ for some collection $F_i\in \ca{F}$, that is, $\up^L c=\up^L \bve_i \bwe F_i$. As $\up^L:C\to \opp{\Fi(L)}$ is a left adjoint, it preserves all joins, and so $\up^L \bve_i \bwe F_i=\bca_i \up^L \bwe F$. For the converse, suppose that $C^*\se \ca{I}(\ca{F}^*)$. For $c\in C$, $\up^L c=\bca_i \up^L \bwe F_i$ for some collection $F_i\in \ca{F}$. Again, by preservation of joins of $\up^L$, we obtain $c=\bve_i \bwe F_i$. To see the equivalence stated in the second claim, we observe that for any filter $F\se L$ we always have $F\se \up^L \bwe F$. For any collection $\ca{F}\se \Fi(L)$ it is the case that for all $F\in \ca{F}$ the reverse set inclusion holds if and only if the Raney extension is $\ca{F}$-compact. But this is also equivalent to having that all filters in $\ca{F}$ are fixpoints of $\up^L\dashv \bwe$, i.e. them being elements of $C^*$.
\end{proof}

\begin{corollary}\label{c:C*canonical}
    For any collection of filters $\ca{F}\se \Fi(L)$, a Raney extension $(L,C)$ such that $C^*\se \ca{I}(\ca{F}\cap C^*)$ is $\ca{F}$-dense.
\end{corollary}
 Existence and uniqueness of what we called $\ca{F}$-canonical extensions of lattices to complete lattices are well-known, and these results stem from the theory of polarities by Birkhoff (see \cite{birkhoff40}). For a general version of the existence and uniqueness results, see for instance Section 2 of \cite{gehrke06}, see \cite{gehrke01} for its application to distributive lattices. From particularizing the analysis of \cite{gehrke01} to the case where we start from a frame, we directly obtain the following.

\begin{theorem}\label{gehrkethm}(see for example \cite{gehrke01}, in particular Remark 2.8)
    For a frame $L$ and a collection $\ca{F}\se L$ of its filters such that $\ca{I}(\ca{F})$ contains the principal ones, there is a unique injective monotone map $f^{\ca{F}}:L\ra L^{\ca{F}}$ to a complete lattice $L^{\ca{F}}$ which is $\ca{F}$-canonical. Concretely, this is the embedding $L\se \opp{\ca{I}(\ca{F})}$ mapping each element to its principal filter. This embedding also preserves the frame operations, and $L$ meet-generates $\opp{\ca{I}(\ca{F})}$.
\end{theorem}

We now wish to adapt the theorem above to prove existence and uniqueness of $\ca{F}$-canonical Raney extensions on a frame $L$ for certain collections of filters $\ca{F}$. 
\begin{lemma}\label{sufficient}
Suppose that $(L,C)$ is a Raney extension. Then:
\begin{itemize}
    \item $C^*$ contains all principal filters; 
    \item All filters in $C^*$ are strongly exact.
\end{itemize}
\end{lemma}
\begin{proof}
For the first item, we only notice that it is clear that $a=\bwe\up^L a$ for all $a\in L$. For the second, suppose that $F\in C^*$, and that $x_i\in F$ is a family such that the meet $\bwe^L_i x_i$, as calculated in $L$, is strongly exact. By definition of Raney extension, this meet is preserved by the embedding $e:L\se C$. This means that $\bwe^L_i x_i=\bwe_i x_i$, where the second meet is computed in $C$. Therefore, since $\bwe F\leq x_i$ for all $i\in I$, we also have $\bwe F\leq \bwe_i x_i$. Since $F\in C^*$, $F=\up^L \bwe F$, and so $\bwe^L_i x_i\in F$.
\end{proof}

\begin{theorem}\label{containsprincipal}
    For a frame $L$ and any collection $\ca{F}\se \Fi(L)$ of filters, the pair $(L,\opp{\ca{I}(\ca{F})})$ is a Raney extension if and only if:
    \begin{enumerate}
    \item $\ca{I}(\ca{F})$ contains all principal filters;
    \item $\ca{I}(\ca{F})^{op}\se \Fi(L)^{op}$ is a subcolocale inclusion;
    \item All filters in $\ca{F}$ are strongly exact.
    \end{enumerate}
    In case these hold, $(L,\opp{\ca{I}(\ca{F})})$ is the unique (up to isomorphism) $\ca{F}$-canonical Raney extension.
\end{theorem}
\begin{proof}
  Let us show that the three conditions are necessary. By Proposition \ref{C*canonical}, if an $\ca{F}$- canonical extension $(L,C)$ exists then $C^*=\opp{\ca{I}(\ca{F})}$. Necessity then follows by Lemma \ref{sufficient}. Let us now show that for a collection $\ca{F}\se \Fi(L)$ satisfying the three properties above, the pair $(L,\opp{\ca{I}(\ca{F})})$ is a Raney extension. We know from Theorem \ref{gehrkethm} that $L\se \opp{\ca{I}(\ca{F})}$ preserves the frame operations (and this is also easy to check), and that $L$ meet-generates the coframe component. We show that the embedding $L\se \opp{\ca{I}(\ca{F})}$ preserves strongly exact meets. Suppose that $x_i\in L$ is a family such that their meet $\bwe^L_i x_i$ is strongly exact. As all filters in $\ca{I}(\ca{F})$ are strongly exact, any such filter which contains $\up x_i$ for all $i\in I$ must also contain $\bwe^L_ix_i$. This means that in the coframe $\ca{I}(\ca{F})^{op}$ the greatest lower bound of the family $\{\up x_i:i\in I\}$ is the principal filter $\up \bwe^L_ix_i$. This means that the meet $\bwe^L_i x_i$ is preserved. The fact that it satisfies the required universal property follows from the characterization in Proposition \ref{C*canonical}, and uniqueness follows from Theorem \ref{gehrkethm}.
\end{proof}

Item (3) of Theorem \ref{containsprincipal} above tells us that for every Raney extension $(L,C)$ there is the upper bound $C^*\se \fse(L)^{op}$. There also is a lower bound.

\begin{lemma}\label{principalmin}
For a frame $L$ the collection $\fe(L)$ is the smallest sublocale of $\Fi(L)$ containing all the principal filters.
\end{lemma}
\begin{proof}
Let $\ca{S}\se \Fi(L)$ be a sublocale containing all the principal filters. For any $x,y\in L$, we must have $\up x\ra \up y\in \ca{S}$. As sublocales are closed under all meets, all intersections of filters of the form $\up x\ra \up y$ must be in $\ca{S}$. Therefore, by the characterization in Lemma \ref{charexact}, $\fe(L)\se \ca{S}$.
\end{proof}

We provide the frame version of a result in \cite{ball14b}: in Theorem 3.7, it is shown that for a meet-semilattice $S$ the smallest frame generated by it is $\ca{J}^{e}(S)$, the collection of all downsets which are closed under those joins of $S$ that distribute over all finite meets. Recently, the same result has been re-proven for frames with bases of meet-semilattices in \cite{bezhanishvili24}.

\begin{proposition}\label{raneymin}
For a frame $L$ the pair $(L,\opp{\fe(L)})$ is a Raney extension, and $\fe(L)^{op}\se C^*$ for all Raney extensions $(L,C)$.
\end{proposition}
\begin{proof}
Principal filters are exact as they are closed under all meets. By Theorem \ref{posetofsubloc}, $\fe(L)\se \Fi(L)$ is a sublocale inclusion and all exact filters are strongly exact. Furthermore, by Lemma \ref{charexact}, $\fe(L)$ is closed under all intersections. Then, $(L,\opp{\fe(L)})$ is a Raney extension by Theorem \ref{containsprincipal}. If $(L,C)$ is a Raney extension, the collection $C^*\se \Fi(L)$ is a sublocale which contains all principal filters, by Lemma \ref{sufficient}, and so $\fe(L)\se C^*$ by Lemma \ref{principalmin}.
\end{proof}

We may order Raney extensions over some frame $L$ by subcolocale inclusion of the coframe components. We obtain a result which may be seen as a version for Raney extensions of Theorem 3.7 in \cite{ball14ext}, where the authors consider the ordered collections of all frames join-generated by a distributive lattice. In the recent work \cite{bezhanishvili24} the result is given a new proof.

\begin{theorem}\label{Rboundaries}
    For a frame $L$, the ordered collection of Raney extensions over $L$ is the interval $[\fe(L),\fse(L)]$ of the coframe of sublocales of $\fse(L)$.
\end{theorem}
\begin{proof}
    That every Raney extension belongs to the section $[\fe(L),\fse(L)]$ follows from Theorem \ref{containsprincipal} and Proposition \ref{raneymin}. Suppose that there is a sublocale $\ca{F}\se \fse(L)$ such that $\fe(L)\se \ca{F}$. By Lemma \ref{principalmin}, $\ca{F}$ contains all principal filters, and so, by Theorem \ref{containsprincipal}, the pair $(L,\opp{\ca{F}})$ is a Raney extension.
\end{proof}

\subsection{Notable examples of Raney extensions}\label{ss:concr}
In this subsection, we look at some concrete examples of Raney extensions.
\begin{proposition}
The following are all Raney extensions.
    \begin{itemize}
    \item The pair $(L,\opp{\fse(L)})$ for any frame $L$;
    \item The pair $(L,\opp{\fe(L)})$ for any frame $L$;
    \item The pair $(L,\opp{\fr(L)})$ for subfit $L$;
    \item The pair $(L,\opp{\ca{I}(\fso(L))})$ for pre-spatial $L$;
    \item The pair $(L,\opp{\ca{I}(\fcp(L))})$ for spatial $L$.
\end{itemize}
\end{proposition}
\begin{proof}
    We use the characterization in Theorem \ref{containsprincipal}. That the collections of filters below are subcolocales of $\Fi(L)$ and that all filters in these collections are strongly exact follows from Theorem \ref{posetofsubloc}. Since principal filters are closed under all meets, they are exact and strongly exact. For the last three items we refer to Proposition \ref{famouschar}.
\end{proof}

Note that all the Raney extensions above are $(L,\opp{\ca{I}(\ca{F})})$ for some $\ca{F}\se \Fi(L)$. This means that each extension $(L,\opp{\ca{I}(\ca{F})})$ above is the unique $\ca{F}$-canonical one. Because of the isomorphisms in Theorem \ref{eandse}, for any frame $L$ the following embeddings into coframes are Raney extensions, up to isomorphism.
\begin{itemize}
    \item $\op:L\to \So(L)$, 
    \item $\cl:L\to \opp{\Sc(L)}$.
\end{itemize}

\section{Topological duality for Raney extensions}\label{S2}

In this section, we show that there is an adjunction between $\opp{\bd{Raney}}$ and $\bd{Top}$. For any coframe $C$, we define $\rpt(C)$ to be the collection of its completely join-prime elements. For a Raney extension $(L,C)$, let us define the function $\varphi_{(L,C)}:C\ra \ca{P}(\rpt(C))$ as
\[
\varphi_{(L,C)}(a)=\{x\in \rpt(C):x\leq a\}.
\]
It is easy to see that the following two properties hold:
\begin{enumerate}
    \item $\varphi_{(L,C)}(\bwe_i a_i)=\bca_i\varphi_{(L,C)}( a_i)$,
    \item $\varphi_{(L,C)}(\bve_i a_i)=\bigcup_i \varphi_{(L,C)}(a_i)$,
\end{enumerate}
for each family $a_i\in L$. When the Raney extension $(L,C)$ is clear from the context, we will omit the subscript. By property 2, the elements of the form $\varphi_{(L,C)}(a)$ for $a\in L$ form a topology. We denote the topological space obtained by equipping the set $\rpt(C)$ with this topology as $\rpt(L,C)$, and we call it the \emph{spectrum} of the Raney extension $(L,C)$. Since all elements of $C$ are meets of elements of $L$, from property 1 it follows that the elements of the form $\varphi_{(L,C)}(c)$ with $c\in C$ are the saturated sets of this space. Let us show functoriality of the assignment $(L,C)\mapsto \rpt(L,C)$. Observe that the following is a pointfree version of Lemma \ref{l:upOmXcp}.

\begin{lemma}\label{cpcjp}
For a Raney extension $(L,C)$, an element $x\in C$ is completely join-prime if and only if $\up^{L} x$ is a completely prime filter.
\end{lemma}
\begin{proof}
It is immediate that if $x\in C$ is completely join-prime then $\up^L x$ is completely prime. For the converse, suppose that $x\in C$ is such that $\up^{L}x$ is completely prime. Suppose that $x\leq \bve D$ for $D\se C$. This means that $\up^L \bve D\se \up^L x$. Observe that $\up^L \bve D=\bigcap \{\up^L d:d\in D\}$. As $\up^L x$ is assumed to be completely prime, there must be some $d\in D$ such that $\up^L d\se \up^L x$. This implies that $x\leq d$.
\end{proof}

\begin{lemma}\label{respectscjp}
For a morphism $f:(L,C)\ra (M,D)$ of Raney extensions, if $x\in \rpt(D)$ then $f^*(x)\in \rpt(C)$. 
\end{lemma}
\begin{proof}
By Lemma \ref{cpcjp}, it suffices to show that for a morphism $f:(L,C)\ra (M,D)$ of Raney extensions, if $x\in \rpt(D)$ then $\up^L f^*(x)$ is a completely prime filter of $L$. If $f^*(x)\leq \bve A$ for $A\se L$, then as $f$ respects the frame operations of $L$, and because $f^*\dashv f$, $x\leq \bve \{f(a):a\in A\}$. Since $x$ is completely join-prime, there is some $a\in A$ such that $x\leq f(a)$, that is $f^*(x)\leq a$.
\end{proof}

\begin{lemma}
The assignment $\rpt:(L,C)\mapsto \rpt(L,C)$ is the object part of a functor $\rpt:\opp{\bd{Raney}}\ra \bd{Top}$ which acts on morphisms as $f\mapsto f^*$.
\end{lemma}
\begin{proof}
That every morphism is mapped to a well-defined function between the set of points follows from Lemma \ref{respectscjp}. Continuity follows from the fact that the $f^*$-preimage of some $\varphi(a)$ for $a\in L$ is, expanding definitions,
\begin{gather*}
   \{x\in \rpt(D):f^*(x)\leq a\}=\\
   \{x\in \rpt(D):x\leq f(a)\}=\varphi(f(a)),
\end{gather*}
and this set is indeed open in $\rpt(D)$ as by definition of Raney morphism $f(a)\in M$.
\end{proof}

By Theorem \ref{charC*}, we may identify Raney extensions with collections of filters. Let us now see how to describe the spectrum under this identification. 

\begin{theorem}\label{t:pointsfilters}
    For a frame $L$ and for a sublocale $\ca{F}\se \Fi(L)$ such that it contains all principal filters, $\pt_R(\ca{F}^{op})=\fcp(L)\cap \ca{F}$. 
    
\end{theorem}
\begin{proof}
   We show that an element $P\in \ca{F}$ is completely prime in the frame $\ca{F}$ if and only if it is completely prime as an element of $\Fi(L)$. If an element $P\in \ca{F}$ is completely prime in the frame $\Fi(L)$, then it is also completely prime as an element of $\ca{F}$, as meets of elements of $\ca{F}$ are a subset of all the meets in $\Fi(L)$. For the converse, suppose that $P$ is completely prime in $\ca{F}$, and that $\bve_i x_i\in P$ for some collection $x_i\in L$. This means $\bca_i \up x_i\se P$, and because $\ca{F}$ contains all principal filters and by assumption on $P$, $\up x_i\se P$ for some $i\in I$. 
\end{proof}

\begin{corollary}\label{c:pointsfilters}
   A Raney extension $(L,C^*)$ has as points the elements of $C^*\cap \fcp(L)$, and as opens the sets of the form $\{P\in \fcp(L)\cap C^*:a\in P\}$ for some $a\in L$.
\end{corollary}
\begin{proof}
    The first part of the statement is a direct consequence of Theorem \ref{t:pointsfilters}. For the second part of the statement, it suffices to unravel the definition of the topology on $\rpt(L,C^*)$.
\end{proof}

We now define the left adjoint to $\rpt$. For a topological space $X$ we define $\Om_R(X)$ as the pair $(\Om(X),\ca{U}(X))$, we extend the assignment to morphisms as $f\mapsto f^{-1}$.

\begin{lemma}\label{raneyrefl}
For every Raney extension $(L,C)$ there is a surjective map of Raney extensions $\varphi_{(L,C)}:(L,C)\ra \Om_R(\rpt(L,C))$. This is an isomorphism precisely when $C$ is join-generated by its completely join-prime elements.
\end{lemma}
\begin{proof}
    The fact that it is a surjection and a map of Raney extensions follows from properties 1 and 2 of the topologizing map $\varphi_{(L,C)}$. The map is an isomorphism precisely when it is injective, and this happens exactly when for $c,d\in C$ such that $c\nleq d$ there is some $x\in \pt_R(C)$ such that $x\leq c$ and $x\nleq d$. This holds if and only if the completely join-prime elements join-generate $C$.
\end{proof}

The map we have just defined will be the evaluation at an object of the natural transformation $\Om_R\circ \rpt\Rightarrow 1_{\opp{\bd{Raney}}}$. Let us now define the other natural transformation $1_{\bd{Top}}\Rightarrow \rpt\circ \Om_R$.
\begin{lemma}\label{toprefl}
For every topological space $X$ the map $\psi_X:X\ra \rpt(\Om_R(X))$ defined as $x\mapsto \up x$ is a continuous map. This is a homeomorphism precisely when $X$ is a $T_0$ space.
\end{lemma}
\begin{proof}
That the map is well-defined and surjective follows from the observation that the completely join-prime elements of $\ca{U}(X)$ are precisely the principal upsets. For continuity, we observe that the $\psi_X$-preimage of an open set $\varphi(U)$ is the set $\{x\in X:\up x\in \varphi(U)\}=U$. This map is also open, as the direct image of an open $U\se X$ is the open $\{\up x:\up x\se U\}=\varphi(U)$. The map is then a homeomorphism when it is injective, and this holds if and only if, whenever $x\neq y$, $\up x\neq \up y$. This amounts to the specialization preorder being an order, that is, the space being $T_0$. 
\end{proof}

Recall that an adjunction $L:\ca{C}\lra \ca{D}:R$ is said to be \emph{idempotent} if every element of the form $R(d)$ for some object $d\in \mf{Obj}(\ca{D})$ is a fixpoint on the $\ca{C}$ side, and the same holds for the $\ca{D}$ side. 
\begin{theorem}
    The pair $(\Om_R,\rpt)$ constitutes an idempotent adjunction $\bd{Top}\lra \opp{\bd{Raney}}$. Raney duality is the restriction of this adjunction to a dual equivalence.
\end{theorem}
\begin{proof}
    The proof of adjointness amounts to standard computations. The two maps in Lemmas \ref{raneyrefl} and \ref{toprefl} are the counit and the unit, respectively. Let us see that the adjunction is idempotent. By Lemma \ref{raneyrefl}, any Raney extension $(\Om(X),\ca{U}(X))$ is a fixpoint, as the coframe $\ca{U}(X)$ is join-generated by the elements of the form $\up x$. By Lemma  \ref{toprefl}, any $T_0$ space is a fixpoint.
\end{proof}

Motivated by the result above and by Lemma \ref{raneyrefl}, we say that a Raney extension $(L,C)$ is \emph{spatial} if $C$ is join-generated by the completely join-prime elements. 

\begin{proposition}\label{charspatiality}
A Raney extension $(L,C)$ is spatial if and only if $C^*\se \ca{I}(C^*\cap \fcp(L))$.
\end{proposition}
\begin{proof}
Because of the isomorphism $\up^L:C\cong C^*$, a Raney extension $(L,C)$ is spatial precisely when all elements of $C^*$ are intersections of completely join-prime elements in $C^*$, as by Corollary \ref{c:pointsfilters}, $\rpt(C^*)=\fcp(L)\cap C^*$.
\end{proof}

\begin{corollary}\label{Xcpdense}
Spatial Raney extensions are $\ca{CP}$-dense.
\end{corollary}
\begin{proof}
This follows from Corollary \ref{c:C*canonical}, and Proposition \ref{charspatiality}.
\end{proof}

\subsection{The collection of Raney spectra on a frame}

In this subsection, our final goal is proving that, on a frame $L$, for any Raney extension $(L,C)$ there are subspace inclusions $\pt_D(L)\se \rpt(L,C)\se \pt(L)$.

\begin{lemma}\label{bigmeet}
For a frame $L$, for any $a\in L$ the meet $\bwe \{x\in L:a<x\}$ is exact.
\end{lemma}
\begin{proof}
Let $L$ be a frame and let $a\in L$. Let us consider the meet $\bwe \{x\in L:a<x\}$. Let $b\in L$. We claim that $\bwe \{x\ve b:a<x\}\leq \bwe \{x\in L:a<x\}\ve b$. We consider two cases. First, let us assume that $b\leq a$. If this is the case, then $b\leq x$ whenever $a<x$, and so both the left hand side and the right hand side equal $\bwe \{x\in L:a<x\}$. Now, let us assume instead that $b\nleq a$. This is equivalent to saying that $a<a\ve b$. This means that:
\[
\bwe \{x\ve b:a<x\}\leq a\ve b\leq \bwe \{x\in L:a<x\}\ve b.\qedhere
\]
\end{proof}
\begin{proposition}\label{p:cpexact}
A completely prime filter $L{\sm}\da p$ is exact if and only if the prime $p$ is covered.
\end{proposition}
Suppose that the completely prime filter $L{\sm}\da p$ is exact. To show that the prime $p$ is covered, we prove that $\bwe \{x\in L:p<x\}\nleq p$. By Lemma \ref{bigmeet}, the meet on the left-hand side is exact. The result follows by our assumption that $L{\sm}\da p$ is closed under exact meets. For the converse, we suppose that $p$ is a covered prime and that $x_i\nleq p$ for the members of some family $\{x_i:i\in I\}$ such that their meet is exact. We then have that $x_i\ve p\neq p$ for every $i\in I$, and as $p$ is covered, this implies that $\bwe_i (x_i\ve p)\neq p$. By exactness of the meet $\bwe_i x_i$, we also have $(\bwe_i x_i)\ve p\neq p$, that is $\bwe _i x_i\nleq p$, as required.

\begin{lemma}\label{spectrumofef}
For any frame $L$, the spectrum of $(L,\opp{\fe(L)})$ is homeomorphic to the space $\pt_D(L)$. The spectrum of $(L,\opp{\fse(L)})$ is the classical spectrum $\pt(L)$.
\end{lemma}
\begin{proof}
By Corollary \ref{c:pointsfilters}, the points of $(L,\opp{\fe(L)})$ are the completely prime filters which are also exact. By Proposition \ref{p:cpexact}, these are the filters of the form $L{\sm}\da p$ for some covered prime $p\in L$. Indeed, then, there is a bijection between the points of $\rpt(L,\opp{\fe(L)})$ and those of $\pt_D(L)$. This is a restriction and co-restriction of the standard homeomorphism between the spectrum $\pt(L)$ and its space of completely prime filters, and so it is a homeomorphism. For $(L,\opp{\fse(L)})$, it suffices to notice that since all completely prime filters are strongly exact, $\fcp(L)\cap \fse(L)=\fcp(L)$.
\end{proof}
We shall now refine the result above to the case of subfit frames. We call $\maxpt(L)$ the collection of maximal primes of a frame $L$, equipped with the subspace topology inherited from $\pt(L)$.

\begin{proposition}\label{minimalcpf}
    Let $L$ be a frame. A prime $p\in L$ is maximal if and only if $L{\sm}\da p$ is a regular filter.
\end{proposition}

\begin{proof}
Suppose that $p\in L$ is a maximal prime. Because it is maximal, $\up p=\{p,1\}$. We claim that the completely prime filter $L{\sm}\da p$ is its pseudocomplement in the frame of filters. Indeed, $L{\sm}\da p\cap \{1,p\}=\{1\}$. Furthermore, if, for a filter $F$, $F\cap \{1,p\}=\{1\}$ then $p\notin F$, and so for $f\in F$ we must have $f\nleq p$. For the converse, suppose that $p\in L$ is such that $L{\sm}\da p$ is a regular filter. By Proposition \ref{reg}, this is the intersection of a collection of filters of the form $\{x\in L:x\ve a=1\}$ for some $a\in L$. As $L{\sm}\da p$ is completely prime, it must be $\{x\in L:x\ve a=1\}$ for some $a\in L$. This means that for all $x\in L$ the conditions $x\leq p$ and $x\ve a\neq 1$ are equivalent. In particular, because the filter is not all of $L$ (as it is completely prime), we must have $a\leq p$ since $a\ve a=a\neq 1$. This means that if $x\nleq p$ then $x\ve a=1$ and so $x\ve p=1$, for all $x\in L$. This means that $p$ must be maximal. 
\end{proof}

\begin{proposition}
For a subfit frame $L$, the spectrum of the Raney extension $(L,\opp{\fr(L)})$ is the $T_1$ space $\maxpt(L)$.
\end{proposition}
\begin{proof}
Suppose that $L$ is a subfit frame. We claim that all its exact filters are regular. By Proposition \ref{famouschar}, $\fr(L)$ contains all principal filters, and so by Lemma \ref{principalmin} we must have $\fe(L)\se \fr(L)$. The reverse inclusion holds for all frames. By Corollary \ref{c:pointsfilters}, then, the points of $(L,\opp{\fe(L)})$ are the regular completely prime filters, which by Proposition \ref{minimalcpf} are those corresponding to maximal primes of $L$. The fact that this is a homeomorphism comes from the fact that this is a restriction of the standard homeomorphism between the spectrum $\pt(L)$ and the spectrum defined in terms of prime elements of $L$. The space $\maxpt(L)$ is a $T_1$ space, since whenever $p,q\in \maxpt(L)$, both $p\nleq q$ and $q\nleq p$ by maximality, and so the open set $\{a\in L:a\nleq p\}$ contains $q$ and omits $p$, and the open set $\{a\in L:a\nleq q\}$ contains $p$ and omits $q$. 
\end{proof}

\begin{lemma}\label{spectrumboundary}
For a Raney extension $(L,C)$ there are subspace embeddings 
\[
\pt_D(L)\to \rpt(L,C)\to \pt(L).
\]  
\end{lemma}
\begin{proof}
If $(L,C)$ is a Raney extension, $\fe(L)\se C^*\se \fse(L)$, by Theorem \ref{containsprincipal} and by Proposition \ref{raneymin}. Therefore, 
\[
\fcp(L)\cap \fe(L)\se\fcp(L)\cap C^*\se\fcp(L)\cap \fse(L).
\]
By Corollary \ref{c:pointsfilters}, this means that thre is a chain of subspace inclusions $\rpt(L,\opp{\fe(L)})\se \rpt(L,C)\se \rpt(L,\opp{\fse(L)})$. The result follows from Lemma \ref{spectrumofef}.
\end{proof}

\begin{lemma}\label{meetofheyting}
For a frame $L$ and a subset $\ca{X}\se \Fi(L)$ the smallest sublocale $\ca{S}(\ca{X})$ is the set $\ca{I}(\{\up a\ra F:a\in L,F\in \ca{X}\})$.
\end{lemma}
\begin{proof}
By Lemma \ref{l:small-sl}, it suffices to show that the collection in the claim is the same as $\ca{I}(\{G\ra F:G\in \Fi(L),F\in \ca{X}\})$. Indeed, for each $G\in \Fi(L)$ and $F\in \ca{X}$, $G\ra F=\bca \{\up g\ra F:g\in G\}$.
\end{proof}
The following fact follows directly from Lemma \ref{l:prime}, and the fact that completely prime filters are prime elements of $\Fi(L)$.
\begin{lemma}\label{l:heytingcp}
    For a frame $L$ and a completely prime filter $P\se L$, for each $a\in L$, 
\[
\up a\ra P=
\begin{cases}
L & \text{ if $a\in P$}\\
P & \text{ otherwise.}
\end{cases}
\]
\end{lemma}

\begin{theorem}\label{spectrumpowerset}
The spectra of Raney extensions over $L$ coincide, up to homeomorphisms, with the interval
\[
[\pt_D(L),\pt(L)] 
\]
of the powerset of $\pt(L)$.
\end{theorem}
\begin{proof}
The spectrum of a Raney extension $(L,C)$ is contained in the $[\pt_D(L),\pt(L)]$ interval by Lemma \ref{spectrumboundary}. To show the converse, because of Lemma \ref{spectrumofef}, it suffices to show that for any collection of completely prime filters $\ca{P}$ such that $\fe(L)\se \ca{P}$ there is some Raney extension $(L,C)$ such that its points are $\ca{P}$, that is, $C^*\cap \fcp(L)=\ca{P}$. Let then $\ca{P}$ be such a collection. Consider the sublocale $\ca{S}(\ca{P}\cup L)\se \Fi(L)$. By Lemma \ref{principalmin}, this is the same as $\ca{S}(\ca{P}\cup \fe(L))$. Observe that $\fe(L)$ is stable under $\up a\to -$ for each $a\in L$, as it is a sublocale. The same holds for $\ca{P}$, by Lemma \ref{l:heytingcp}. By Lemma \ref{meetofheyting}, $\ca{S}(\ca{P}\cup L)=\ca{I}(\ca{P}\cup \fe(L))$. We now consider the Raney extension $(L,\opp{\ca{S}(\ca{P}\cup L)})$. It is clear that all the elements of $\ca{P}$ are points of this Raney extension, by Corollary \ref{c:pointsfilters}. Let us show the reverse set inclusion. Suppose that there is a completely prime filter $F$ such that $F\in \ca{S}(\ca{P}\cup L)$. By complete primality, and by the characterization above, this is either in $\ca{P}$ or in $\fe(L)$. In the second case, it is in $\ca{P}$, too, by assumption on $\ca{P}$. Indeed, then, $\ca{S}(\ca{P}\cup L)\cap \fcp(L)=\ca{P}$, as desired.
\end{proof}

\section{Topological properties and Raney extensions}\label{S3}

\subsection{Sobriety}

We now look at sobriety and characterize it in terms of Raney extensions.
\begin{lemma}[\cite{jakl24}, Lemma 5.4]\label{mcpupsets}
For every frame $L$ there is an isomorphism $\iota:\ca{I}(\fcp(L))^{op}\cong\ca{U}(\pt(L))$ defined on generators as $\iota(P)=\up P=\{Q\in \fcp(L):P\se Q\}$.   
\end{lemma}

\begin{lemma}\label{l:upOmXcp}
    Let $X$ be a $T_0$ space. For a saturated set $Y\se X$, the filter $\up^{\Om(X)}Y$ is completely prime if and only if $Y=\up x$ for some $x\in X$.
\end{lemma}
\begin{proof}
For all $y\in Y$, as $Y$ is saturated, $\up y\se Y$. Towards contradiction, suppose that $Y\nsubseteq \up y$ for each $y\in Y$. For each $y\in Y$, let $y'\in Y{\sm}\up y$, so $y\nleq y'$, and let $U_y\in \Om(X)$ be such that $y\in U_y$ and $y'\notin U_y$. Then, $Y\se \bigcup \{U_y:y\in Y\}$. But $Y\nsubseteq U_y$ for all $y\in Y$ as $y'\in Y{\sm}U_y$. This contradicts complete primality of $\up^{\Om(X)}Y$.
\end{proof}

\begin{proposition}\label{p:charsober}
 A $T_0$ topological space $X$ is sober if and only if $(\Om(X),\ca{U}(X))$ is $\ca{CP}$-compact. In particular, a $T_0$ space $X$ is sober if and only if the pair $(\Om(X),\ca{U}(X))$ is a realization of the $\ca{CP}$-canonical Raney extension.
\end{proposition}
\begin{proof}
Suppose that $X$ is a sober space, and let $P$ be a completely prime filter. By sobriety, we may assume $P=N(x)$ for some $x\in X$. Since $\bca N(x)=\up x$, indeed, $\bca N(x)\se U$ implies that $x\in U$, and so $U\in N(x)$. Then, $(\Om(X),\ca{U}(X))$ is $\ca{CP}$-compact. For the second part of the claim, assume that for some $T_0$ space $X$ the pair $(\Om(X),\ca{U}(X))$ is $\ca{CP}$-compact. We show sobriety by showing that any completely prime filter $P$ is a neighborhood filter. By the characterization in Proposition \ref{C*canonical}, this is of the form $\up^{\Om(X)}Y$, for some saturated set $Y$. By Lemma \ref{l:upOmXcp}, $Y=\up x$ for some $x\in X$. Hence $P=N(x)$. The last part of the claim follows by combining this characterization of sobriety with Corollary \ref{Xcpdense}. 
\end{proof}
Motivated by this result, we define a Raney extension $(L,C)$ to be \emph{sober} if it is $\ca{CP}$-compact. In Section \ref{S4} we will show the existence of sober coreflections in $\bd{Raney}$. Let us now compare sobriety with spatiality for Raney extensions.

\begin{lemma}\label{cpcanonical}
    A Raney extension $(L,C)$ is sober and spatial if and only if it is $\ca{CP}$-canonical.
\end{lemma}
\begin{proof}
It follows from Proposition \ref{charspatiality} and by Proposition \ref{C*canonical} that a Raney extension $(L,C)$ is sober and spatial if and only if $C^*=\ca{I}(\fcp(L))$. This holds if and only if the Raney extension is $\ca{CP}$-canonical. 
\end{proof}

\begin{proposition}
For a spatial frame $L$, $(L,\opp{\ca{I}(\fcp(L))})$ is the unique (up to isomorphism) sober and spatial Raney extension.
\end{proposition}
\begin{proof}
   By Lemma \ref{cpcanonical}, when a sober and spatial Raney extension exists, it is unique, up to isomorphism, by Theorem \ref{containsprincipal}. If $L$ is a spatial frame, then $(L,\opp{\ca{I}(\fcp(L))})$ is a Raney extension by Proposition \ref{famouschar}, and it is the $\ca{CP}$-canonical Raney extension by Theorem \ref{containsprincipal}.
\end{proof}

\subsection{The \texorpdfstring{T\textsubscript{D}}{TD} axiom}

Let us now look at the Raney analogue of the $T_D$ axiom. 

\begin{lemma}\label{l:N(x)exact}
    A $T_0$ space is $T_D$ if and only if all neighborhood filters are exact.
\end{lemma}
\begin{proof}
    Suppose that $X$ is a $T_D$ topological space. Neighborhood filters are completely prime, and by Proposition \ref{allpxarecov} all primes of the form $X{\sm}\overline{\{x\}}$ are covered. Hence, by the characterization in Proposition \ref{p:cpexact}, the corresponding neighborhood filters are exact. Conversely, if $X$ is not $T_D$ there must be a point $x\in X$ whose prime is not covered, and by Proposition \ref{p:cpexact} again, this means that its completely prime filter is not exact.    
\end{proof}

\begin{theorem}\label{t:chartd}
The following are equivalent for a $T_0$ space $X$.
\begin{enumerate}
    \item The space $X$ is $T_D$. 
    \item The Raney extension $(\Om(X),\ca{U}(X))$ is $\ca{E}$-dense.
    \item The Raney extension $(\Om(X),\ca{U}(X))$ is $\ca{E}$-canonical.
    \item The Raney extension $(\Om(X),\ca{U}(X))$ is isomorphic to $(\Om(X),\fe(\Om(X)))$.
    \item The inclusion $\Om(X)\se \ca{U}(X)$ preserves exact meets.
\end{enumerate}
\end{theorem}
\begin{proof}
Let $X$ be a $T_0$ space. If this is a $T_D$ space, then by Lemma \ref{l:N(x)exact} all neighborhood filters are exact, and this means that all filters of the form $\up^{\Om(X)}\up x$ for $x\in X$ are exact. As, for all $x\in X$, $\bca \up^{\Om(X)}\up x=\up x$, and the principal filters generate the collection $\ca{U}(X)$, (2) follows. Suppose, now, that (2) holds. By Proposition \ref{raneymin}, and the characterization of compactness of Proposition \ref{C*canonical}, the Raney extension is $\ca{E}$-compact, hence $\ca{E}$-canonical by our initial hypothesis. Items (3) and (4) are equivalent by the uniqueness result of Theorem \ref{containsprincipal}. Suppose that (4) holds. We will identify $(\Om(X),\ca{U}(X))$ with the isomorphic Raney extension $(\Om(X),\opp{\fe(\Om(X))})$. If $U_i\in \Om(X)$ is a family such that their meet is exact, then the least upper bound of the family $\up U_i$ in $\fe(\Om(X))$ must be $\up \bwe_i U_i$, by definition of exact filter. This means that the meet is preserved by the embedding $\Om(X)\to \opp{\fe(\Om(X))}$. Finally, (5) implies (1) by the characterization in Theorem \ref{t:ball}.
\end{proof}

Motivated by the last result, we call a Raney extension $T_D$ if it is $\ca{E}$-dense. All Raney extensions are $\ca{E}$-compact, by Proposition \ref{raneymin} and the characterization of compactness in Proposition \ref{C*canonical}. Thus, the $T_D$ Raney extensions are those which are $\ca{E}$-canonical, and by the uniqueness result of Theorem \ref{containsprincipal} these are the Raney extensions which are, up to isomorphism, $(L,\opp{\fe(L)})$ for some frame $L$. In Section \ref{S4}, we will study $T_D$ Raney extensions.

\subsection{The \texorpdfstring{T\textsubscript{1}}{T1} axiom}
Let us now look at the $T_1$ axiom. The axiom $T_1$, too, can be characterized in terms of filters.
\begin{lemma}\label{l:N(x)regular}
    A $T_0$ space is $T_1$ if and only if all its neighborhood filters are regular.
\end{lemma}
\begin{proof}
    Suppose that $X$ is a $T_1$ space, and let $x\in X$. As $X$ is $T_1$, the set $X{\sm}\{x\}$ is open. Then $N(x)=\{U\in \Om(X):U\cup (X{\sm}\{x\})=X\}$. By the characterization of regular filters in Proposition \ref{reg}, this is a regular filter. For the converse, suppose that $X$ is a $T_0$ space where all neighborhood filters are regular. Let $x\in X$. We will show that $\{x\}$ is closed by showing $\da x=\{x\}$. By the characterization in Proposition \ref{reg}, and because neighborhood filters are completely prime, there is some open $V\in \Om(X)$ such that: 
\[
N(x)=\{U\in \Om(X):U\cup V=X\}.
\]
Observe that $\bca N(x)\cup V=\up x\cup V=X$, thus $V^c\se \up x$. Since $\emptyset \notin N(x)$, $V\cup P=\emptyset \cup V\neq X$. Then, $V\notin N(x)$ and so $x\in V^c$. As $V^c$ is a downset in the specialization order, $\da x\se V^c$. But this means $\da x\se V^c\se \up x$, hence $\da x=\{x\}$. 
\end{proof}

\begin{lemma}\label{l:t1P(x)}
    If $(\Om(X),\ca{P}(X))$ is a Raney extension then it is $(\Om(X),\opp{\fr(\Om(X))})$, up to isomorphism.
\end{lemma}
\begin{proof}
    Let $X$ be a space. We observe that if $(\Om(X),\ca{P}(X))$ is a Raney extension, each subset is an intersection of opens, by meet generativity, and so $X$ is $T_1$. To show our claim, it suffices to show that $\ca{P}(X)^*=\fr(\Om(X))$. Filters in $\ca{P}(X)^*$ are those of the form $\up^{\Om(X)}S$ for arbitrary subsets $S\se X$. As $X$ is $T_1$, $S$ is closed, so for each $S\se X$, $\up^{\Om(X)}S=\{U\in \Om(X):S^c\cup U=X\}$. Since $S^c$ is open, this filter is regular by Proposition \ref{reg}.
\end{proof}

\begin{theorem}
The following are equivalent for a $T_0$ space $X$.
\begin{enumerate}
    \item The space $X$ is $T_1$. 
    \item The pair $(\Om(X),\ca{P}(X))$ is a Raney extension, and equals $(\Om(X),\ca{U}(X))$.
    \item The Raney extension $(\Om(X),\ca{U}(X))$ is $\ca{R}$-dense.
    \item The Raney extension $(\Om(X),\ca{U}(X))$ is $\ca{R}$-canonical.
    \item The Raney extension $(\Om(X),\ca{U}(X))$ is isomorphic to $(\Om(X),\fr(\Om(X)))$.
    
\end{enumerate}
\begin{proof}
The equivalence between (1) and (2) holds as $\Om(X)$ meet-generates $\ca{P}(X)$ if and only if $X$ is $T_1$, and if $\Om(X)\se \ca{P}(X)$ meet-generates it, the other properties of Raney extensions are easy to check. If (2) holds, then $\ca{R}$-density by Lemma \ref{l:t1P(x)}, and by the characterization in Proposition \ref{C*canonical}. By Proposition \ref{raneymin}, and since every regular filter is exact, any Raney extension is $\ca{R}$-compact, thus (3) implies (4). If (4) holds, then (5) follows from the uniqueness result in Theorem \ref{containsprincipal}. If (4) holds, then $\ca{U}(X)^*=\fr(\Om(X))$, and so all neighborhood filters are regular. By Lemma \ref{l:N(x)regular}, (1) follows.
\end{proof}

\end{theorem}
Let us then study the $T_1$ axiom more pointfreely. A space $X$ is $T_1$ if and only if $\ca{U}(X)$ is the same as the powerset $\ca{P}(X)$. Motivated by this, we define a Raney extension $(L,C)$ to be $T_1$ if and only if $C$ is a Boolean algebra.

\begin{theorem}\label{charsubfit}
    For a frame $L$, the following are equivalent.
    
    \begin{enumerate}
        \item $L$ is subfit.
        \item All exact filters of $L$ are regular.
        \item $(L,\opp{\fe(L)})$ is a $T_1$ Raney extension.
        \item There exists a $T_1$ Raney extension $(L,C)$.
        \item There is a unique $T_1$ Raney extension on $L$, up to isomorphism. This is $(L,\opp{\fr(L)})$.
        \end{enumerate}
\end{theorem}
\begin{proof}
    Suppose that $L$ is a subfit frame. By Proposition \ref{famouschar}, all principal filters are regular filters. By Lemma \ref{principalmin}, this implies that $\fe(L)\se \fr(L)$. Now, suppose that $\fe(L)\se \fr(L)$. This implies that $(L,\opp{\fe(L)})=(L,\opp{\fr(L)})$, as regular filters are exact for every frame. By Proposition \ref{reg}, the coframe $\opp{\fr(L)}$ is a Boolean algebra. It is clear that condition (3) implies condition (4). Let us show that (4) implies (5). If $(L,B)$ is a Raney extension such that $B$ is Boolean, $\fe(L)\se B^*$ by Proposition \ref{raneymin}. As $\fe(L)\se \Fi(L)$ is dense, $B^*$ is dense, too, and as the only sublocale that is both Boolean and dense is the Booleanization this means $\fr(L)=\fe(L)=B^*$. Thus, $(L,B)$ and $(L,\opp{\fr(L)})$ are isomorphic. Now, suppose that (5) holds. Then, $(L,\opp{\fr(L)})$ is a Raney extension. This means that all principal filters are regular, and so by Proposition \ref{famouschar} the frame $L$ must be subfit.
\end{proof}

\subsection{Scatteredness}
The notion of scattered space is already present in classical topology, see for example \cite{willard70}. In \cite{simmons80} it is proven that a $T_0$ space is scattered if and only if $\mf{S}(\Om(X))$ is Boolean. This motivates the definition of scattered frame (see \cite{plewe00}): a frame is \emph{scattered} if the coframe $\mf{S}(L)$ is Boolean. As proven in \cite{ball16}, a frame is scattered and subfit if and only if $\mf{S}(L)=\Sc(L)$. Subfit scattered frames are also fit, and so $\mf{S}(L)=\mf{S}_{\op}(L)$.

\begin{proposition}
For a subfit frame $L$, the following are equivalent.
\begin{enumerate}
    \item The frame $L$ is scattered.
    \item $\fse(L)=\fe(L)=\fr(L)$.
    \item $\fse(L)=\fe(L)$.
     \item The frame has a unique Raney extension, up to isomorphism.
    \item $\mf{S}_{\op}(L)=\mf{S}_{\cl}(L)$.
    \item The frame has a unique Raney extension, up to isomorphism, and this is $(L,\sll)$.
\end{enumerate}
\end{proposition}
\begin{proof}
Suppose that $L$ is a scattered subfit frame. Let $F$ be a strongly exact filter, by Theorem \ref{eandse} this is $\{x\in L:S\se \op(x)\}$ for some sublocale $S$. By hypothesis, $S$ is a join $\bve_i \cl(x_i)$ of closed sublocales, so that 
\[
F=\{x\in L:x_i\ve x=1\text{ for all }i\in I\}=\bca_i \{x\in L:x\ve x_i=1\}.
\]
By the characterization of regular filters in Proposition \ref{reg}, then, $\fse(L)\se \fr(L)$. This implies (2), as for all frames $\fr(L)\se \fe(L)\se \fse(L)$. It is clear that (2) implies (3). Let us show that (3) implies (4). The inclusion $\fe(L)\se \fse(L)$ holds for every frame. Now, suppose that in $L$ every strongly exact filter is exact. For any Raney extension $(L,C)$, we must have $\fe(L)\se C^*\se \fse(L)$. Our assumption, then, implies $\fe(L)=C^*=\fse(L)$. Suppose, now, that $L$ has a unique Raney extension, up to isomorphism. The pair $(L,\So(L))$ is a Raney extension. As $L$ is subfit, this must be a Boolean extension, by Theorem \ref{charsubfit}. As $\mf{S}_{\op}(L)$ is a subcoframe of $\mf{S}(L)$, this means that in $\mf{S}(L)$ every fitted sublocale has a complement, which is itself a fitted sublocale. In particular, all joins of closed sublocales are fitted and so $\mf{S}_{\cl}(L)\se\mf{S}_{\op}(L)$. Finally, recall that the lattice $\mf{S}_b(L)$ of joins of complemented sublocales is $\mf{S}_{\cl}(L)$ for subfit frames. We then also have the reverse set inclusion $\mf{S}_{\op}(L)\se \mf{S}_{\cl}(L)$. If (5) holds, by subfitness $(L,\mf{S}_{\op}(L))$ is a Boolean extension. Since this is the largest Raney extension, all its Raney extensions must be Boolean. By Theorem \ref{charsubfit}, when Boolean extensions exist, they are unique. Note also that $\So(L)=\Sc(L)$ implies that every closed sublocale is fitted, and this implies that the frame $L$ is subfit, hence $\So(L)=\mf{S}(L)$. Suppose, finally, that (6) holds. Because all subfit frames have a Boolean extension, by Theorem \ref{charsubfit}, $\sll$ must be Boolean, and so $L$ is scattered.
\end{proof}

\section{Free and cofree constructions}\label{S4}
\subsection{Extensions of frame maps}
We ask when a map $f:L\to M$ of frames can be extended to Raney extensions of these frames. In general, for Raney extensions $(L,C)$ and $(M,D)$, if such $\overline{f}:C\to D$ exists it has to be defined as $\overline{f}(c)=\bwe \{f(a):a\in \up^L c\}$.

\begin{theorem}\label{whenliftsfilters}
    A frame map $f:L\to M$ extends to a map $\overline{f}:(L,C)\to (M,D)$ between Raney extensions if and only if $f^{-1}(F)\in C^*$ for every $F\in D^*$.
\end{theorem}
\begin{proof}
First, suppose that there is a map $\overline{f}:(L,C)\to (M,D)$ of Raney extensions extending $f$. In particular, $\overline{f}$ preserves all meets and so it has a left adjoint $\overline{f}^*$. Consider $d\in D$. By adjointness, $f^{-1}(\up^M d)=\up^L \overline{f}^*(d)$. For the converse, suppose $f^{-1}(F)\in C^*$ for every $F\in D^*$. This means that for every $d\in D$ there is some $\overline{f}^*(d)\in C$ such that $f^{-1}(\up^M d)=\up^L \overline{f}^*(d)$, which is unique as $L$ meet-generates $C$. First, we claim that $\overline{f}^*:D\to C$ preserves all joins. To show $\bigvee_i \overline{f}^*(d_i)=\overline{f}^*(\bigvee_i d_i)$, it suffices to show $\up^L \bigvee_i \overline{f}^*(d_i)=\up^L \overline{f}^*(\bigvee_i d_i)$.
\begin{align*}
\up^L \bigvee_i \overline{f}^*(d_i)&=\bigcap_i \up^L\overline{f}^*(d_i)=\bigcap_i f^{-1}(\up^L d_i)=f^{-1}(\bca_i \up^L d_i)=f^{-1}(\up^L \bve_i d_i)=\up^L \overline{f}^*(\bve_i d_i).
\end{align*}
Then, $\overline{f}^*:D\to C$ has a right adjoint, which we call $\overline{f}$. We claim that this is the required map. First, we show that it extends $f$. For $a\in L$, 
\begin{align*}
    \overline{f}(a)&=\bve \{d\in D:\overline{f}^*(d)\leq a\}=\bigvee\{d\in D:a\in \up^L\overline{f}^*(d)\}\\
    & =\bigvee\{d\in D:a\in f^{-1}(\up^L d)\}=\bigvee\{d\in D:d\leq f(a)\}=f(a).
\end{align*}
As it is a right adjoint, it preserves all meets. For finite joins, consider two arbitrary elements $\bwe_i a_i,\bwe_jb_j\in C$ with $a_i,b_j\in L$. Then 
\begin{align*}
\overline{f}(\bwe_i a_i\ve \bwe_j b_j)=\bwe_{i,j}\overline{f}(a_i\ve b_j)=\bwe_{i,j}\overline{f}(a_i)\ve \overline{f}(b_j)=\overline{f}(\bwe_i a_i)\ve \overline{f}(\bwe_j b_j),    
\end{align*}
where we have used that $\overline{f}$ preserves joins of $L$ as well as coframe distributivity. Then, the map $\overline{f}$ is the sought for map.
\end{proof}
In the coming subsections, we use this theorem to construct various free and cofree objects.

 \subsection{Free Raney extension on a frame}
 
There is a forgetful functor $\pi_1:\bd{Raney}\ra \bd{Frm}$ which forgets about the second component of the extension. We will show that this has a left adjoint.
\begin{lemma}\label{selifts}
    Any frame morphism $f:L\ra M$ extends to a Raney morphism
    \[
    f_{\ca{SE}}:(L,\opp{\fse(L)})\ra (M,\opp{\fse(M)}).
    \]
\end{lemma}
\begin{proof}
By Theorem \ref{whenliftsfilters}, a frame morphism $f:L\ra M$ lifts as required if preimages of strongly exact filters are strongly exact. Suppose, then, that $F\se M$ is strongly exact. Suppose that the meet $\bwe_i x_i$ is strongly exact, and that $f(x_i)\in F$. Because all frame morphisms preserve strongly exact meets, as well as strong exactness of meets, $\bwe_i f(x_i)=f(\bwe_i x_i)\in F$, as desired.
\end{proof}

Then, there is a functor $\fse:\bd{Frm}\to \bd{Raney}$ assigning to a frame $L$ the Raney extension $(L,\opp{\fse(L)})$, and to a morphism $f:L\to M$ the Raney map $f_{\ca{SE}}:(L,\opp{\fse(L)})\to (M,\opp{\fse(M)})$ extending it, which exists by Lemma \ref{selifts}.

\begin{theorem}\label{raneymax}
 For a frame $L$, the pair $(L,\opp{\fse(L)})$ is the free Raney extension over it, that is, $\Fi_{\ca{SE}}\dashv \pi_1$. In particular, the category of frames is a full coreflective subcategory of $\bd{Raney}$.
\end{theorem}
\begin{proof}
 Suppose that $f:L\ra M$ is a frame map. Let $(M,D)$ be a Raney extension. By Lemma \ref{sufficient}, $D^*\se \opp{\fse(M)}$, and by Lemma \ref{selifts}, preimages of strongly exact filters are strongly exact. Therefore, preimages of elements in $D^*$ are in $\fse(L)$. By Theorem \ref{whenliftsfilters}, the frame map lifts to a map $\overline{f}:(L,\opp{\fse(L)})\to (M,D)$ extending $f$.
\end{proof}

\subsection{Free spatial Raney extension on a spatial frame}

It is known that, for all frame maps, preimages of completely prime filters are completely prime. Recall that the preimage map also preserves arbitrary intersections. This, together with Theorem \ref{whenliftsfilters}, gives us the following.
\begin{lemma}\label{cplifts}
Any frame morphism $f:L\ra M$ between spatial frames extends to a Raney morphism
\[
f_{\ca{CP}}:(L,\opp{\ca{I}(\fcp(L))})\ra (M,\opp{\ca{I}(\fcp(M))}).
\]
\end{lemma}
The lemma above shows that there is a functor $\fcp:\bd{spFrm}\to \bd{Raney}$ from the category of spatial frames, assigning to a frame $L$ the Raney extension $(L,\opp{\fcp(L)})$, and to a morphism $f:L\to M$ the Raney map $f_{\ca{CP}}:(L,\opp{\fcp(L)})\to (M,\opp{\fcp(M)})$ extending it.
\begin{proposition}
    For a spatial frame $L$, the pair $(L,\opp{\ca{I}(\fcp(L))})$ is the free spatial Raney extension over it. In particular, the category of spatial frames is a full coreflective subcategory of that of spatial Raney extensions.
\end{proposition}
\begin{proof}
Suppose that $f:L\ra M$ is a map between spatial frames, and that $(M,C)$ is a spatial Raney extension. By spatiality, we must have $C^*\se \ca{I}(\fcp(M))$, by Proposition \ref{charspatiality}. Preimages under $f$ of completely prime filters are completely prime. This means that preimages of filters in $C^*$ are in $\ca{I}(\fcp(L))$. By Theorem \ref{whenliftsfilters}, there is a morphism $(L,\opp{\ca{I}(\fcp(L))})\ra(M,C)$ which extends the frame map $f:L\ra M$. 
\end{proof}

\subsection{Cofree Raney extension for frames with exact maps}
Since for any Raney extension $(L,C)$ there is a subcolocale inclusion $\fe(L)\se C^*$, there is a Raney surjection $(L,C)\to (L,\opp{\fe(L)})$. In light of Theorem \ref{raneymax} it is natural to wonder if $L\mapsto(L,\opp{\fe(L)})$ is the object part of a right adjoint to $\pi_1:\bd{Raney}\to \bd{Frm}$. This is not the case. As shown in \cite{ball19} -- where the question is explored for structure $\Sc(L)$, isomorphic to $\fe(L)$ -- not all frames maps $f:L\to M$ can be extended to their coframes of exact filters. For a frame morphism $f:L\ra M$, we will say that it is \emph{exact} if whenever the meet of a family $\{x_i:i\in I\}\se L$ is exact, so is the meet of $\{f(x_i):i\in I\}$, and furthermore $\bwe_i f(x_i)=f(\bwe_i x_i)$.

\begin{proposition}\label{whenelifts}
A morphism $f:L\ra M$ is exact if and only if preimages of exact filters are exact. This holds if and only if the morphism can be extended to a Raney morphism 
\[
f_{\ca{E}}:(L,\opp{\fe(L)})\ra (M,\opp{\fe(M)}).
\]
\end{proposition}
\begin{proof}
Suppose that $f:L\ra M$ is an exact frame map, and that $G\se M$ is an exact filter. Suppose that $\bwe_i x_i\in L$ is an exact meet such that $f(x_i)\in G$. By exactness of this map, the meet $\bwe_i f(x_i)$ is exact and so $\bwe_i f(x_i)\in G$. Again, by exactness of $f$, $\bwe_i f(x_i)=f(\bwe_i x_i)$. Indeed, then, $\bwe_i x_i\in f^{-1}(G)$. Conversely, suppose that there is a frame map $f:L\ra M$ that it is not exact. This means that either there is an exact meet $\bwe_i x_i\in L$ such that it is not preserved by $f$, or there is an exact meet $\bwe_i x_i\in L$ such that $\bwe_i f(x_i)$ is not exact. We consider these two cases in turn. In the first case, we consider the principal filter $\up \bwe_i f(x_i)$. This is exact, as it is closed under all meets. We notice that by our hypothesis $f(\bwe_i x_i)$ is not an element of this filter. Let us call $F$ the preimage of this filter. Both $x_i\in F$ and $\bwe_i x_i\notin F$, and so $F$ is not exact. In the second case, consider an exact meet $\bwe_i x_i\in L$ such that $\bwe_i f(x_i)$ is not exact. In particular, let $y\in M$ be such that $\bwe_i (f(x_i)\ve y)\nleq (\bwe_i f(x_i))\ve y$. We now consider the exact filter 
\[
\up y \ra \up\bwe_i (f(x_i)\ve y)=\{m\in M:\bwe_i (f(x_i)\ve y)\leq y\ve m\}.
\]
That this is an exact filter follows from the characterization of Lemma \ref{charexact}. Let $F$ be the preimage of this filter. For each $i\in I$, $x_i\in F$. We claim that $\bwe_i x_i\notin F$. This follows from the fact that by our hypothesis $\bwe_i (f(x_i)\ve y)\nleq (\bwe_i f(x_i))\ve y$ and $f(\bwe_i x_i)\leq \bwe_i f(x_i)$. The rest of the claim follows by Theorem \ref{whenliftsfilters}.
\end{proof}

Let us call $\bd{Frm}_{\ca{E}}$ the category of frames with exact maps. Let us also call $\bd{Raney}_{\ca{E}}$ the category of Raney extensions with morphisms $f$ such that $\pi_1(f)$ is exact. The assignment $L\mapsto (L,\opp{\fe(L)})$, then, determines a functor $\Fi_{\ca{E}}:\bd{Frm}_{\ca{E}}\to \bd{Raney}_{\ca{E}}$.

\begin{theorem}
   There is an adjunction $\pi_1:\bd{Raney}_{\ca{E}}\lra \bd{Frm}_{\ca{E}}:\Fi_{\ca{E}}$ with $\pi_1\dashv \Fi_{\ca{E}}$.
\end{theorem}
\begin{proof}
Suppose that $f:L\ra M$ is an exact frame map, and that $(L,C)$ is a Raney extension. By Proposition \ref{whenelifts}, as $f$ is exact, preimages of filters in $\fe(M)$ are in $\fe(L)$. Furthermore, $\fe(L)\se C^*$ by Proposition \ref{raneymin}. Then, preimages of exact filters of $M$ are in $C^*$. By Theorem \ref{whenliftsfilters}, then, there is a map of Raney extensions $(L,C)\ra (M,\opp{\fe(M)})$ extending $f$, as desired. 
\end{proof}

Recall that we have defined the Raney extensions of the form $(L,\opp{\fe(L)})$ as the $T_D$ ones. Let $\bd{Raney_D}$ be the full subcategory of $\bd{Raney}$ given by the $T_D$ Raney extensions.
\begin{proposition}
    The forgetful functor $\pi_1:\bd{Raney}\to \bd{Frm}$ restricts to an isomorphism $\bd{Frm}_{\ca{E}}\cong \bd{Raney_D}$.
\end{proposition}
\begin{proof}
    For a map $f:(L,\opp{\fe(L)})\to (M,\opp{\fe(L)})$ of $T_D$ Raney extensions, by Theorem \ref{whenliftsfilters} the restriction $f|_{L}:L\to M$ is a map in $\bd{Frm}_{\ca{E}}$. Thus, the restriction and co-restriction of $\pi_1$ is well-defined. The inverse functor maps each frame $L$ to the Raney extension $(L,\opp{\fe(L)})$, and this assignment is functorial by Proposition \ref{whenelifts}.
\end{proof}

In contrast with the frame setting, the $T_D$ objects form a full subcategory of our pointfree category. We will explore the consequences of this in relation to the $T_D$ duality in Subsection \ref{specialTD}

\subsection{Sober coreflection of a Raney extension}
Theorem \ref{whenliftsfilters}, more broadly, can also be used to construct coreflections and reflections.
For a Raney extension $(L,C)$ we call a map $\sigma:S(L,C)\ra (L,C)$ of the category $\bd{Raney}$ a \emph{sobrification} if $S(L,C)$ is sober, and if whenever $f:(M,D)\ra (L,C)$ is a morphism from a sober Raney extension, there a commuting diagram
\[
\begin{tikzcd}
S(L,C)
\ar[r,"\sigma"]
& (L,C).\\
(M,D)
\ar[u,"f_{\sigma}"]
\ar[ur,swap,"f"]
\end{tikzcd}
\]

\begin{theorem}
For a Raney extension $(L,C)$, the map
\begin{gather*}
\sigma:(L,\opp{\ca{I}(C^*\cup \fcp(L))})\ra (L,C)\\
F\mapsto \bwe F
\end{gather*}
is its sobrification.
\end{theorem}
\begin{proof}
Observe that, as $C^*\se \Fi(L)$ is a sublocale and by Lemmas \ref{meetofheyting} and \ref{l:heytingcp}, $\ca{I}(C^*\cup \fcp(L))$ is a sublocale. As $(L,\opp{\ca{I}(C^*\cup \fcp(L))})$ contains all completely prime filters of $L$, indeed, by Proposition \ref{C*canonical} it is $\ca{CP}$-compact. Since $C^*\se \ca{I}(C^*\cup \fcp(L))$, by Theorem \ref{whenliftsfilters} it means that the identity on $L$ extends to a surjective map of Raney extensions
\begin{gather*}
    \sigma:(L,\opp{\ca{I}(C^*\cup \fcp(L))})\to (L,C)\\
    F\mapsto \bwe F.
\end{gather*}
Let us show that this map has the required universal property. Suppose that $f:(M,D)\ra (L,C)$ is a Raney map from a sober Raney extension. Consider the frame map $f|_{M}:M\ra L$. By Theorem \ref{whenliftsfilters}, to show that the map lifts it suffices to show that the preimage of each filter in $\fcp(L)$ as well as each filter in $C^*$ is in $D^*$. For filters in $C^*$, this holds because there is a map $f:(M,D)\ra (L,C)$. For a completely prime filter $P\se L$, recall that $f^{-1}(P)\in \fcp(M)$, and by definition of sobriety and Proposition \ref{C*canonical}, also $\fcp(M)\se D^*$. 
\end{proof}

\subsection{\texorpdfstring{T\textsubscript{D}}{TD} reflection for Raney extensions and exact maps}

We have seen that in $\bd{Raney}$, in contrast to $\bd{Frm}$, one can speak of $T_D$ objects that are not necessarily spatial. This enables us to define a notion of $T_D$ reflection of an object which does not spatialize said object. Once again, the morphisms have to be restricted. For a Raney extension $(L,C)$, we call a \emph{$T_D$ reflection} a map $\delta:(L,C)\to D(L,C)$ such that $D(L,C)$ is $T_D$, and such that whenever $f:(L,C)\to (M,D)$ is a map to a $T_D$ Raney extension, these is a commuting diagram as follows.
\[
\begin{tikzcd}
(L,C)
\ar[r,"\delta"]
\ar[dr,"f",swap]
& D(L,C)
\ar[d,"f_{\delta}"]\\
& (M,D).
\end{tikzcd}
\]

\begin{proposition}
    In the category $\bd{Raney}_{\ca{E}}$, every Raney extension admits a $T_D$ reflection.
\end{proposition}
\begin{proof}
   We claim that the required map for $(L,C)$ is the surjection $\delta:(L,C)\to (L,\opp{\fe(L)})$ following from Proposition \ref{raneymin}. This is a map in $\bd{Raney}_{\ca{E}}$, as it restricts to the identity on $L$. Now, suppose that there is a $T_D$ Raney extension $(M,D)$ such that there is a morphism $f:(L,C)\to (M,D)$ in $\bd{Raney}_{\ca{E}}$. Because $(M,D)$ is $T_D$, by definition $D^*=\opp{\fe(L)}$. By assumption on $f$, then, the preimage map relative to $f|_{L}$ maps filters in $D^*$ to exact filters of $L$. Hence, by Theorem \ref{whenliftsfilters}, there is a map $f_{\delta}:(L,\opp{\fe(L)})\to (M,D)$ as required. Finally, this map is in $\bd{Raney}_{\ca{E}}$ as it extends $f|_L$.
\end{proof}

\subsection{Canonical extension as free algebraic Raney extension on a pre-spatial frame}\label{s:pit}

We now view the canonical extension of a frame from \cite{jakl20} as a Raney extension, and characterize it as a free construction. For a pre-spatial frame $L$, we will call its \emph{canonical extension} the Raney extension $(L,\opp{\ca{I}(\fso(L))})$. For this pair to be a Raney extension, we do need pre-spatiality, by Proposition \ref{famouschar}.

\begin{lemma}\label{solifts}
    Any frame morphism $f:L\ra M$ between pre-spatial frames extends to a Raney morphism
    \[
    f_{\ca{SO}}:(L,\opp{\ca{I}(\fso(L))})\ra (M,\opp{\ca{I}(\fso(M))}).
    \]
\end{lemma}
\begin{proof}
By Theorem \ref{whenliftsfilters}, it suffices to show that for a frame morphism $f:L\ra M$ between pre-spatial frames preimages of Scott-open filters are Scott-open. Suppose that $F\se L$ is a Scott-open filter, and that $\{x_i:i\in I\}\se L$ is a directed family such that $f(\bve_i x_i)=\bve f(x_i)\in F$. Observe that the family $\{f(x_i):i\in I\}$ is directed, and so by Scott-openness of $F$ we must have $f(x_i)\in F$ for some $i\in I$, as desired.
\end{proof}

 For a Raney extension $(L,C)$, we say that an element $c\in C$ is \emph{compact} if, for every directed collection $D\se L$, $c\leq \bve D$ implies that $c\leq d$ for some $d\in D$. We say that a Raney extension $(L,C)$ is \emph{algebraic} if every element of $c$ is the join of compact elements. 
 
\begin{lemma}\label{meetofscott}
    A Raney extension $(L,C)$ is algebraic if and only if $C^*\se \ca{I}(C^*\cap \fso(L))$.
\end{lemma}
\begin{proof}
Notice that an element $x\in C$ is compact if and only if the filter $\up^L x$ is Scott-open. Consider the isomorphism $\up^L:C\cong C^*$. The Raney extension $(L,C)$ is algebraic if and only if in $C^*$ every element is a join of Scott-open filters of the form $\up^L x$ for some $x\in L$. The inclusion $C^*\se \opp{\Fi(L)}$ is a subcolocale inclusion, and subcolocale inclusions preserves all joins, and joins in $\opp{\Fi(L)}$ are intersections. Therefore, algebraicity of $(L,C)$ is equivalent to every filter in $C^*$ being an intersection of Scott-open filters in $C^*$.
\end{proof}

\begin{lemma}\label{prespatial}
A frame admits an algebraic Raney extension if and only if it is pre-spatial.
\end{lemma}
\begin{proof}
First, we observe that if a frame admits an algebraic Raney extension this means that principal filters must all be intersections of Scott-open filters, by Lemma \ref{meetofscott}. By Proposition \ref{famouschar}, the frames with this property are exactly the pre-spatial ones. For a pre-spatial frame $L$, an algebraic Raney extension is $(L,\opp{\ca{I}(\fso(L))})$. 
\end{proof}

We are now ready to characterize canonical extensions of frames as free algebraic Raney extensions.
\begin{theorem}
     For a pre-spatial frame $L$, its canonical extension is the free algebraic Raney extension over it.
\end{theorem}
\begin{proof}
Suppose that $L$ is a pre-spatial frame, and that $(M,C)$ is an algebraic Raney extension. Suppose that there is a frame map $f:L\ra M$. Consider the canonical extension $(L,\opp{\ca{I}(\fso(L))})$. As $(M,C)$ is algebraic, $C^*\se \opp{\ca{I}(\fso(M))}$, by Lemma \ref{meetofscott}. By Lemma \ref{solifts}, preimages of Scott-open filters are Scott-open. Then, preimages of filters in $C^*$ are in $\ca{I}(\fso(L))$. By Theorem \ref{whenliftsfilters} this means that there is a map of Raney extensions $(L,\opp{\ca{I}(\fso(L))})\ra (M,C)$ extending the frame map $f:L\ra M$. 
\end{proof}

\section{Special topics}\label{S5}

\subsection{Sobriety and strict sobriety}\label{S51}
We work towards characterizing sobriety and strict sobriety of spaces in terms of Raney extensions.

\begin{lemma}\label{l:ssiffsocompact}
    A $T_0$ space $X$ is strictly sober if and only if $(\Om(X),\ca{U}(X))$ is a $\ca{SO}$-compact Raney extension.
\end{lemma}
\begin{proof}
    By Proposition \ref{C*canonical}, a Raney extension is $\ca{SO}$-compact if and only if $\fso(L)\se C^*$. The claim follows by definition of strict sobriety.
\end{proof}

\begin{proposition}\label{charpostsober}
   A $T_0$ space is strictly sober if and only if $(\Om(X),\ca{U}(X))$ is the canonical extension of $\Om(X)$.  
\end{proposition}
\begin{proof}
If $X$ is strictly sober, $(\Om(X),\ca{U}(X))$ is $\ca{SO}$-compact, by Lemma \ref{l:ssiffsocompact}. It is also $\ca{CP}$-dense, by Corollary \ref{Xcpdense}, and as completely prime filters are Scott-open it is also $\ca{SO}$-dense. Conversely, if $X$ is a space such that $(\Om(X),\ca{U}(X))$ is a canonical extension, in particular this Raney extension is $\ca{SO}$-compact, hence $X$ is strictly sober by Lemma \ref{l:ssiffsocompact}.
\end{proof}

It is known that if we do not assume choice principles it is not the case that sobriety implies strict sobriety. Let us look at a concrete counterexample for this.

\begin{example}\label{counterpit}
We assume the negation of the Ultrafilter Lemma and deduce that there exists a sober space which is not strictly sober. Let $X$ be a set and let $\ca{P}(X)$ be its powerset, let $F\se \ca{P}(X)$ be a filter such that it is not contained in any ultrafilter. Now, consider the Stone dual $X^S$ of $\ca{P}(X)$, let $\varphi$ be its topologizing map. Note that this is an isomorphism of Boolean algebras, as $\ca{P}(X)$ is atomic. All elements of the form $\varphi(Y)$ for $Y\se X$ are clopens of the space $X^S$, hence compact. We then have that the filter of opens $\up \varphi[F]$ is Scott-open. By assumption, $\bca \varphi[F]=\emptyset$, and so if $\up \varphi[F]$ is a neighborhood filter of some compact open. This must be $\emptyset$, but this is not the case as the neighborhood filter of $\emptyset$ contains $\emptyset$, and $\emptyset\notin \varphi[F]$ by injectivity of $\varphi$. Thus, the space $X^S$ is sober, as it is a Stone space, but it is not strictly sober.

\end{example}

We want to rephrase the $\textbf{SPET}$ property in terms of Raney extensions.

\begin{lemma}\label{spet}
The Prime Ideal Theorem is equivalent to the statement that $\fso(L)\se \ca{I}(\fcp(L))$ for every frame $L$.
\end{lemma}
\begin{proof}
We need to show that every Scott-open filter being an intersection of completely prime filters is equivalent to $\textbf{SPET}$. Suppose that $\textbf{SPET}$ holds, and that $L$ is a frame and $F\se L$ a Scott-open filter. Suppose, towards contradiction, that there is some $a\notin F$ such that $a\in P$ whenever $P$ is a completely prime filter with $F\se P$. By $\textbf{SPET}$, there is a prime element $p\in L$ with $a\leq p$ and $p\notin F$. The completely prime filter $L{\sm}\da p$ contains $F$ but not $a$, and this is a contradiction. Conversely, suppose that every Scott-open filter is in $\ca{I}(\fcp(L))$. Let $F\se L$ be a Scott-open filter, and suppose that $a\notin F$. There has to be a prime $p\in L$ such that $F\se L{\sm}\da p$ and such that $a\notin L{\sm}\da p$.
\end{proof}

\begin{proposition}
    The following are equivalent.
    \begin{enumerate}
        \item The Prime Ideal Theorem holds.
        \item $\fso(L)\se \ca{I}(\fcp(L))$ for every frame $L$.
        \item $\ca{CP}$-compact Raney extensions are $\ca{SO}$-compact.
        \item Sober spaces are strictly sober.
        \item For a sober space $X$, the canonical extension of its frame of opens is $(\Om(X),\ca{U}(X))$.
    \end{enumerate}
\end{proposition}
\begin{proof}
    That (1) and (2) are equivalent follows from Lemma \ref{spet}. If (2) holds, (3) follows by the characterization in Proposition \ref{C*canonical}. Suppose, now, that (3) holds. For a sober space $X$, the Raney extension $(\Om(X),\ca{U}(X))$ is $\ca{CP}$-compact, by Proposition \ref{p:charsober}. Therefore, $(\Om(X),\ca{U}(X))$ is $\ca{SO}$-compact, by hypothesis, hence strictly sober, by the characterization in Lemma \ref{l:ssiffsocompact}. Items (4) and (5) are equivalent by Proposition \ref{charpostsober}. Finally, (4) implies (1) by Example \ref{counterpit}.
\end{proof}

Finally, we give another proof, based on Raney extensions, of the result in \cite{bezhanishvili22} that the canonical extension of a Boolean algebra $B$ is the Booleanization of $\ca{U}(\mf{Idl}(B))$.

\begin{proposition}\label{distrce}(\cite{jakl20}, Proposition 8.1)
   For a coherent frame $L$, its canonical extension is the canonical extension of the distributive lattice $K(L)$ of its compact elements.
\end{proposition}

\begin{lemma}\label{l:complF}
For a frame $L$, if $k\in L$ is a complemented element, then in the frame $\Fi(L)$ the filters $\up k$ and $\up \neg k$ are mutual complements. 
\end{lemma}
\begin{proof}
Since $\neg \up k=\{a\in L:a\ve k=1\}$, $\neg \up k=\up \neg k$. By definition of complement, also $\up k\cap \up \neg k=\{1\}$ and $\up k\ve \up \neg k=L$.
\end{proof}

\begin{lemma}\label{l:soinzd}
For a compact, zero-dimensional frame $L$, Scott-open filters are exactly the joins of filters of the form $\up k$, where $k\in L$ is a complemented element.
\end{lemma}
\begin{proof}
    Let $L$ be a compact, zero-dimensional frame, and let $F$ be a Scott-open filter. Let $f\in F$, and let $\{k_i:i\in I\}$ be the family of complemented elements below it, Then, $F=\bve_i k_i$. Observe that this is directed, and so there must be $j\in I$ with $k_j\in F$. For the converse, if $K$ is any family of complemented elements, suppose that there is a directed family $D\se L$ such that $k\leq \bve D$ for some $k\in D$. As $L$ is compact, $k$ is compact, too, and so $k\leq d$ for some $d\in D$. Then, $\bve_{k\in K}\up k$ is Scott-open. 
\end{proof}

\begin{lemma}\label{cdregisscott}
     For a compact, zero-dimensional frame $L$, $\fr(L)=\ca{I}(\fso(L))$.
\end{lemma}
\begin{proof}
    Because $\fr(L)$ is the Booleanization of $\Fi(L)$, and this is the smallest sublocale containing $\{1\}$, for the inclusion $\fr(L)\se \ca{I}(\fso(L))$ it suffices to show that $\{1\}$ is Scott-open, but this follows immediately from compactness of $L$. For the other direction, it suffices to show that, for a Scott-open filter $F$, $\neg \neg F\se F$. Let $F$ be a Scott-open filter. By Lemma \ref{l:soinzd}, this is $\bve_i \up k_i$ for some collection $k_i\in L$ of complemented elements, which we can assume to be closed under finite meets without loss of generality. The equalities $\neg \neg \bve_i \up k_i=\neg \bca_i \neg \up k_i=\neg \bca_i \up \neg k_i$ hold, where Lemma \ref{l:complF} is used for the last equality. Note also that $\neg \bca_i \up \neg k_i=\neg \up \bve_i \neg k_i$. Now, if $x\in \neg \neg F$ this means that $\bve_i \neg k_i\ve x=1$, and by compactness this means that $\neg k_j\ve x=1$ for some $j\in I$. Therefore $k_j\leq x$, and so $x\in F$.
\end{proof}

    \begin{proposition}\label{cd}
        Let $L$ be a compact, zero-dimensional frame. Its canonical extension is 
        \[
        (L,\opp{\fr(L)}).
        \]
        This is also the canonical extension of the Boolean algebra $K(L)$. 
    \end{proposition}
    \begin{proof}
        The first part of the claim follows from Lemma \ref{cdregisscott}. The second part of the claim follows from Proposition \ref{distrce}, and the fact that $\fr(L)$ is the Booleanization of $\Fi(L)$. This is also the Booleanization of $\ca{U}(L)$, because, for each $x\in L$, $\neg \up x=\{y\in L:x\ve y=1\}$, and so the regular elements of $\ca{U}(L)$ are precisely the intersections of upsets of this form, that this, the regular filters of $L$.
    \end{proof}

\subsection{Exactness and \texorpdfstring{T\textsubscript{D}}{TD} duality}\label{specialTD}

Recall that the category $\bd{Frm}_{\ca{E}}$ is embedded in $\bd{Raney}$ as the full subcategory of $T_D$ objects. In this section, we study $\bd{Frm}_{\ca{E}}$ as a pointfree category of $T_D$ spaces.
\begin{lemma}\label{l:EimpliesD}
    Exact morphisms are D-morphisms.
\end{lemma}
\begin{proof}
    Suppose that $f:L\to M$ is an exact frame map, and let $p\in M$ be a covered prime. By Proposition \ref{p:cpexact}, $L{\sm}\da p$ is exact. By exactness of $f$, so is $f^{-1}(L{\sm}\da p)$. By adjointness, $f^{-1}(L{\sm}\da p)=L{\sm}\da f_*(p)$. By Proposition \ref{p:cpexact} again, $f_*(p)$ is covered.
\end{proof}
\begin{lemma}\label{l:tdimpliesexact}
Any $T_D$ frame map $f:L\to M$ such that $M$ is $T_D$-spatial is exact. 
\end{lemma}
\begin{proof}
   Suppose that $L$ and $M$ are frames and $M$ is $T_D$-spatial, and that there is a frame map $f:L\to M$ such that $f_*(p)$ is a covered prime whenever $p\in M$ is covered. Now, suppose that $\bwe_i x_i\in L$ is an exact meet. We show $\bwe_i f(x_i)\leq f(\bwe_i x_i)$. Suppose that $p\in M$ is a covered prime with $f(\bwe_i x_i)\leq p$. Then $\bwe_i x_i\leq f_*(p)$, that is, $\bwe_i x_i\ve f_*(p)=f_*(p)$. By exactness, $\bwe_i (x_i\ve f_*(p))=f_*(p)$, and by coveredness there is $i\in I$ with $x_i\ve f_*(p)=f_*(p)$. Then, $f(x_i)\leq p$, which implies $\bwe_i f(x_i)\leq p$, and by $T_D$-spatiality this implies $\bwe_i f(x_i)\leq f(\bwe_i x_i)$ as desired. Let us now show that $\bwe_i f(x_i)$ is exact. Let $y\in M$. Suppose that $\bwe_i f(x_i)\ve y\leq p$ for $p\in M$ a covered prime. As shown above, this means $f(\bwe_i x_i)\ve y\leq p$. Similarly as above, we obtain $f(x_i)\ve y\leq p$ for some $i\in I$, and so $\bwe_i (f(x_i)\ve y)\leq p$. By $T_D$-spatiality, $\bwe_i (f(x_i)\ve y)\leq \bwe_i f(x_i)\ve y$, as desired.
   \end{proof}

\begin{theorem}
   There is an adjunction $\Om:\bd{Top}_D\lra \bd{Frm}_{\ca{E}}^{op}:\pt_D$.
\end{theorem}
\begin{proof}
It suffices to show that the functor $\Om$ maps continuous maps between $T_D$ spaces to exact frame maps, and that the $T_D$ spatialization map of a frame is exact. By Lemma \ref{l:tdimpliesexact}, it is known that the spatialization map is a $T_D$ morphism, so by Lemma \ref{l:tdimpliesexact} it is also exact. By the same Lemma, a map $f:X\to Y$ between $T_D$ spaces determines an exact frame map $\Om(f):\Om(Y)\to \Om(X)$.
\end{proof}

This means that $T_D$ duality remains intact if we replace $\bd{Frm}_D$ with the subcategory $\bd{Frm}_{\ca{E}}$. The advantage of working in this category is that the definition of the morphisms does not mention points, and that all morphisms $f:L\to M$ lift to morphisms $\Sc(f):\Sc(L)\to \Sc(M)$, by the isomorphism $\Sc(L)\cong \fe(L)$. The situation is illustrated below, where the functor $\mf{S}_{\cl}$ is the one mapping a frame $L$ to the Raney extension $(L,\opp{\scl})$.
\begin{center}
    \begin{tikzcd}[row sep=large,column sep=large]
        \bd{Frm}_{\ca{E}}
        \ar[r,"\pt_D"]
        \ar[d,"\Sc"]
        & \bd{Top}_D
        \\
       \bd{Raney_D},
    \ar[ur,"\rpt"]
    \end{tikzcd}
\end{center}

We now look at the notion of sublocale in $\bd{Frm}_{\ca{E}}$.

\begin{lemma}\label{l:exstability}
For a frame $L$, if a meet $\bwe_i x_i\in L$ is exact, then so is $\bwe_i (x_i\ve y)$ for all $i\in I$.
\end{lemma}
\begin{proof}
    Observe that, if $\bwe_i x_i$ is exact, for all $z\in L$, $\bwe_i (x_i \ve y \ve z)\leq (\bwe_i x_i)\ve y\ve z\leq (\bwe_i (x_i\ve y))\ve z$.
\end{proof}
\begin{proposition}
A surjective frame map $f:L\to M$ such that it preserves exact meets is exact.
\end{proposition}
\begin{proof}
Suppose that $\bwe_i x_i$ is exact and that $f:L\to M$ is a frame surjection which preserves exact meets. For $u\in L$, $\bwe_i (f(x_i)\ve f(u))=\bwe_i f(x_i\ve u)=f(\bwe_i x_i \ve u)=\bwe_i f(x_i)\ve f(u)$. We have used Lemma \ref{l:exstability} for the first equality. Since all elements of $M$ are $f(v)$ for some $v\in L$, the meet $\bwe_i f(x_i)$ is exact.
\end{proof}

We say that a sublocale is \emph{exact} if the corresponding surjection is exact. Let us call $\mf{S}_{\ca{E}}(L)$ the ordered collection of exact sublocales of a frame.

\begin{proposition}\label{p:charexactS}
    A sublocale $S$ is exact if and only if for every exact meet $\bwe_i x_i$ and, for all $x\in L$, $\cl(x_i)\cap S\se \cl(x)$ for all $i\in I$ implies that $\cl(\bwe_i x_i)\cap S\se \cl(x)$.
\end{proposition}
\begin{proof}
The surjection corresponding to a sublocale $S\se L$ is the map $\sigma_S:x\mapsto \bwe\{s\in S:x\leq s\}$. Meets in $\sigma_S[L]=S$ are computed as $\bwe_i^S \sigma_S(x_i)=\bwe \{s\in S:x_i\leq s\mb{ for some $i\in I$}\}$. Exactness of $S$ amounts to having, for every exact meet $\bwe_i x_i$, that $\bwe \{s\in S:x_i\leq s\mb{ for some }i\in I\}\leq \bwe\{s\in S:\bwe_i x_i\leq s\}$. Observe that we can re-write this as $\bwe (\bcu_i S\cap \up x_i)\leq \bwe (S\cap \up\bwe_i x_i)$. By definition of the closure of a sublocale, and by definition of closed sublocale, this means that the condition is also equivalent to $\mi{cl}(S\cap \cl(\bwe_i x_i))\se \mi{cl}(\bve_i (S\cap \cl(x_i)))$, and this is equivalent to the given condition.
\end{proof}
\begin{remark}
    We note that the result above can be generalized: a sublocale $S$ is such that $\sigma_S$ preserves a certain class of meets if and only if for all meets $\bwe_i x_i$ in that class, for all $x\in L$, that $\cl(x_i)\cap S\se \cl(x)$ for all $i\in I$ implies that $\cl(\bwe_i x_i)\cap S\se \cl(x)$.
\end{remark}

\begin{proposition}\label{p:ex-suff}
The collection $\Se(L)$ is closed under all joins, and it contains

\begin{itemize}
    \item All closed sublocales;
    \item All open sublocales;
    \item The two-element sublocales $\bl(p)$ for covered $p$.
\end{itemize}    
\end{proposition}
\begin{proof}
  By Proposition \ref{p:charexactS}, if $S_j$ is a collection of exact sublocales, and $\bwe_i x_i$ an exact meet, then $\cl(x_i)\cap \bve_j S_j\se \cl(x)$ implies that $\cl(x_i)\cap S_j\se \cl(x)$ for all $j$'s, by Lemma \ref{l:linear}. Therefore, for all $j$'s, $\cl(\bwe_i x_i)\cap S_j\se \cl(x)$, and the result follows again by linearity. To see that it contains all closed sublocales, consider that if $\cl(x_i)\cap \cl(y)\se \cl(x)$ then $\cl(x_i\ve y)\se \cl(x)$, that is $x \leq x_i\ve y$, and so $x\leq \bwe_i x_i\ve y$, by exactness, and this is equivalent to $\cl(\bwe_i x_i)\cap \cl(y)\se \cl(x)$. Finally, for open sublocales, we notice that $\cl(x_i)\cap \op(y)\se \cl(x)$ means $\cl(x_i)\se \cl(y)\ve \cl(x)$, and this, by exactness, means $\cl(\bwe_i x_i)\se \cl(y)\ve \cl(x)$, that is $\cl(\bwe_i x_i)\cap \op(y)\se \cl(x)$, as desired. For the third part, consider a covered prime $p\in L$ and suppose that, for an exact meet $\bwe_i x_i\in L$, $\cl(x_i)\cap \bl(p)\se \cl(x)$. This means that $\bl(p)\se \cl(x)\ve \op(x_i)$ for all $i$'s. Using the properties of prime elements in Lemma \ref{l:prime}, we obtain that either $x\leq p$ or $x_i\nleq p$ for all $i\in I$. In the first case, $\bl(p)\se \cl(x)$, and the desired result follows. In the second case, $\bwe_i (x_i\ve p)=\bwe_i x_i\ve p\neq p$, by exactness and coveredness, and so $\bwe_i x_i\nleq p$, from which the desired claim follows.
\end{proof}

$D$-sublocales, introduced in \cite{arrieta21}, are those sublocales such that the corresponding surjection is in $\bd{Frm}_D$. These may be seen as a (D-)spatial versions of exact sublocales. 
\begin{proposition}
    A D-spatial sublocale is exact if and only if it is a D-sublocale.
\end{proposition}
\begin{proof}
D-spatial sublocales coincide with joins of two-elements sublocales from covered primes. Consider a D-spatial sublocale $\bve_i \bl(p_i)$, for $p_i\in \pt_D(L)$. If this is exact it is also a D-sublocale, by Lemma \ref{l:EimpliesD}. If it is a D-sublocale, then it is exact by Lemma \ref{l:tdimpliesexact}.    
\end{proof}

\begin{corollary}
    A two-element sublocale $\bl(p)$ is exact if and only if $p$ is covered.
\end{corollary}

\begin{lemma}\label{l:SCLsuff}
If a subcollection $\ca{S}\se \Ss(L)$ is closed under joins and is stable under the operation $-\cap \cl(x)$ and $-\cap \op(x)$ for all $x\in L$, then it is a subcolocale.
\end{lemma}
\begin{proof}
Suppose that $\ca{S}\se \Ss(L)$ is closed under all joins and stable under the two operations above. For it to be a subcolocale, it suffices to show that if $S\in \ca{S}$ and $T\in \Ss(L)$ then $S{\sm}T\in \ca{S}$. Every sublocale of $L$ is of the form $\bca_i \op(x_i)\ve \cl(y_i)$, and $S{\sm}\bca_i \op(x_i)\ve \cl(y_i)=\bve_i (S{\sm}(\op(x_i)\ve \cl(y_i)))$. Then, for $\ca{S}$ to be a subcolocale it suffices for it to be stable under $-{\sm}(\op(x)\ve \cl(y))$. If $\ca{S}$ is as required, and $S\in \ca{S}$, and $x,y\in L$, $S\cap \cl(x)\cap \op(y)=S{\sm}(\op(x)\ve \cl(y))\in \ca{S}$.
\end{proof}

\begin{theorem}
    The inclusion $\Se(L)\se \Ss(L)$ is a subcolocale inclusion.
\end{theorem}
\begin{proof}
   By Lemma \ref{l:SCLsuff}, it suffices to show that the collection is closed under all joins and stable under $-\cap \cl(x)$ and $-\cap \op(x)$ for all $x\in L$. The first claim follows from Proposition \ref{p:charexactS}. For the second, suppose that $y\in L$. Suppose that $S$ is exact. We show that $S\cap \op(y)$ is exact. If, for exact $\bwe_i x_i$, $\cl(x_i)\cap S\cap \op(y)\se \cl(x)$, then $\cl(x_i)\cap S\se \cl(x)\ve \cl(y)=\cl(x\we y)$, and so by hypothesis $\cl(\bwe_i x_i)\cap S\se \cl(x\we y)$, that is $\cl(\bwe_i x_i)\cap S\cap \op(y)\se \cl(x)$. Let us show that $S\cap \cl(y)$ is exact. For exact $\bwe_i x_i$, if $\cl(x_i)\cap S\cap \cl(y)\se \cl(x)$ then $\cl(x_i\ve y)\cap S\se \cl(x)$, and since $\bwe_i (x_i\ve y)$ is exact by Lemma \ref{l:exstability}, and by exactness of $\bwe_i x_i$, this implies that $\cl(\bwe_i x_i\ve y)\cap S\se \cl(x)$, that is $\cl(\bwe_i x_i)\cap S\cap \cl(y)\se \cl(x)$.
\end{proof}

For every frame $L$, there are subcolocale inclusions $\Se(L)\se \mf{S}_{D}(L)\se \mf{S}(L)$. We do not know, yet, how to characterize frames for which $\Se(L)=\mf{S}_D(L)$, or those such that $\Se(L)=\mf{S}(L)$, and leave this as an open question.

\printbibliography

\end{document}